\documentclass[final,times,10pt]{elsarticle}

\usepackage[utf8]{inputenc}

\usepackage[english]{babel}

\usepackage{amsmath}
\usepackage{amsfonts}
\usepackage{amssymb}
\usepackage{geometry}

\usepackage{graphicx}
\usepackage{epstopdf}
\usepackage{grffile} 

\usepackage{pdflscape}

\usepackage[numbers]{natbib}
\usepackage{hyperref} 
\usepackage{caption}
\usepackage{subcaption}
\usepackage{float}
\usepackage{xspace}

\usepackage[amsmath,thmmarks]{ntheorem}
\numberwithin{equation}{section}

\newtheorem{theorem}{Theorem}[section]

\newtheorem{remark}[theorem]{Remark}

\newtheorem{corollary}[theorem]{Corollary}

\theoremstyle{nonumberplain}
\theorembodyfont{\normalfont}
\theoremsymbol{\ensuremath{\blacksquare}}
\newtheorem{proof}{Proof}

\usepackage{siunitx}

\usepackage[nameinlink,capitalise,noabbrev]{cleveref}
\crefname{equation}{}{}
\crefrangelabelformat{equation}{(#3#1#4--#5\crefstripprefix{#1}{#2}#6)}

\usepackage{booktabs}

\usepackage{dcolumn}
\newcolumntype{d}{D{.}{.}{-1}}

\usepackage{color}
\definecolor{newgreen}{RGB}{0,85,68}
\definecolor{aggrored}{RGB}{255,47,47}
\definecolor{steelblue3}{RGB}{33,29,132}

\newcommand{\important}[1]{{\color{aggrored}{#1}}}

\usepackage{mathtools}
\DeclarePairedDelimiter\abs{\lvert}{\rvert}
\DeclarePairedDelimiter\norm{\lVert}{\rVert}
\DeclarePairedDelimiterX\set[2]{\{}{\}}{\,#1 \;\delimsize\vert\; #2\,}

\numberwithin{equation}{section}

\newcommand{\ud}{\mathrm{d}}

\newcommand{\IR}{\mathbb{R}}

\renewcommand{\epsilon}{\ensuremath\varepsilon}

\newcommand   {\pd}   [2][]{\ensuremath{\frac{\partial #2}{\partial #1}}}

\newcommand   {\qpd}  [2][]{\pd[#1]{#2}}

\newcommand   {\grad} [2][]{\ensuremath{\nabla_{#1}#2}}
\renewcommand {\div}  [2][]{\ensuremath{\nabla_{#1}\cdot #2}}

\newcommand   {\laplace}{\nabla^2}

\newcommand{\byd}[1]{\;\text{d}#1}

\renewcommand{\vec}[1]{\ensuremath{\boldsymbol{#1}}}

\newcommand{\bbf}{\vec{f}}
\newcommand{\bF}{\vec{F}}
\newcommand{\bg}{\vec{g}}
\newcommand{\bG}{\vec{G}}

\newcommand{\bbs}{\vec{s}}
\newcommand{\bS}{\vec{S}}
\newcommand{\bu}{\vec{u}}
\newcommand{\bU}{\vec{U}}
\newcommand{\bv}{\vec{v}}

\newcommand{\bx}{\vec{x}}

\newcommand{\mmc}{\overline{c}}

\newcommand{\mE}{\overline{E}}

\newcommand{\mq}{\overline{q}}

\newcommand{\mS}{\overline{S}}
\newcommand{\mbS}{\overline{\bS}}
\newcommand{\mbbs}{\overline{\bbs}}

\newcommand{\mbu}{\overline{\bu}}

\newcommand{\mbU}{\overline{\bU}}

\newcommand{\mrho}{\overline{\rho}}
\newcommand{\mrhov}{\overline{\rho v}}
\newcommand{\mrhoe}{\overline{\rho e}}

\newcommand{\cF}{\ensuremath{\mathcal{F}}}
\newcommand{\cG}{\ensuremath{\mathcal{G}}}

\newcommand{\cL}{\ensuremath{\mathcal{L}}}
\newcommand{\cQ}{\ensuremath{\mathcal{Q}}}
\newcommand{\cR}{\ensuremath{\mathcal{R}}}
\newcommand{\cS}{\ensuremath{\mathcal{S}}}

\newcommand{\cW}{\ensuremath{\mathcal{W}}}

\newcommand{\algo}[1]{\textsc{#1}\xspace}

\newcommand{\cwenonaive}{\algo{Cweno3}\xspace}
\newcommand{\cwenoisentropic}{\algo{Cweno3} wb\xspace}

\newcommand{\half}{\ensuremath{1/2}}

\usepackage{overpic}
\newcommand{\ig}[2][]{
  \IfFileExists{#2}
    {\includegraphics[#1]{#2}}
    {\begin{overpic}[#1]{img/dummy.jpg}
       \put (10,85) {\important{\framebox{\detokenize{#2}}}}
     \end{overpic}
     \message{File #2 does not exist.}
    }
}

\begin{document}

\begin{frontmatter}

\title{High-order well-balanced finite volume schemes for the Euler equations
       with gravitation}

\author[auth1]{L.~Grosheintz-Laval \corref{cor1}}
\cortext[cor1]{Corresponding author}
\ead{luc.grosheintz@sam.math.ethz.ch}
\author[auth1]{R.~K\"appeli}

\address[auth1]{
Seminar for Applied Mathematics (SAM),
Department of Mathematics,
ETH Z\"urich,
CH-8092 Z\"urich,
Switzerland}


\begin{abstract}
A high-order well-balanced scheme for the Euler equations with gravitation is
presented.
The scheme is able to preserve a spatially high-order accurate discrete
representation of a large class of hydrostatic equilibria.
It is based on a novel local hydrostatic reconstruction, which, in combination
with any standard high-order accurate reconstruction procedure, achieves
genuine high-order accuracy for smooth solutions close or away from equilibrium.
The resulting scheme is very simple and can be implemented into any existing
finite volume code with minimal effort.
Moreover, the scheme is not tied to any particular form of the equation of
state, which is crucial for example in astrophysical applications.
Several numerical experiments demonstrate the robustness and high-order
accuracy of the scheme nearby and out of hydrostatic equilibrium.
\end{abstract}

\end{frontmatter}


\section{Introduction}
\label{sec:intro}
A multitude of interesting physical phenomena are modeled by the Euler
equations with gravitational source terms.
Applications range from the study of atmospheric phenomena, such as numerical
weather prediction and climate modeling, to the numerical simulation of the
climate of exoplanets, convection in stars and core-collapse supernova
explosions.
The Euler equations with gravitational source terms express the conservation
of mass, momentum and energy:
\begin{align}
  \label{eq:euler_continuity}
  \qpd[t]{\rho} + \div{\left( \rho \bv \right)} &= 0 \\
  \label{eq:euler_momentum}
  \qpd[t]{\rho \bv} + \div{\left( \rho \bv\otimes\bv \right)}
                    + \grad{p} &= -\rho \grad{\phi} \\
  \label{eq:euler_energy}
  \qpd[t]{E} + \div{\left( \bv\,(E + p) \right)} &= -\rho \bv \grad{\phi}.
\end{align}
Here $\rho$ is the mass density, $\bv$ the velocity and
\begin{align}
  \label{eq:intro_0010}
  E = \rho e + \frac{\rho}{2} v^2
\end{align}
the total fluid energy density being the sum of internal and kinetic energy
densities.
The pressure $p$ is related to the density and specific internal energy
through an equation of state $p = p(\rho,e)$.

The source terms on the right-hand side of the momentum and energy equations
model the effect of the gravitational forces on the fluid.
They are dictated by the variation of the gravitational potential $\phi$, which
can either be a given function or, in the case of self-gravity, be
determined by the Poisson equation
\begin{align}
  \label{eq:poisson}
  \laplace{\phi} = 4 \pi G \rho
  ,
\end{align}
where $G$ is the gravitational constant.

In many physically relevant applications, such as the ones named above, (parts
of) the flow of interest may be realized close to hydrostatic equilibrium
\begin{align}
  \label{eq:equilibrium}
  \grad{p} = -\rho \grad{\phi}.
\end{align}
As a matter of fact, the numerical simulation of near equilibrium flows is
challenging for standard finite volume methods.
The reason for this is that these methods may in general not satisfy a
discrete equivalent of the equilibrium.
Thus such states are not preserved exactly but are solely approximated with
an error proportional to the truncation error of the scheme.
So if the interest relies in the simulation of small perturbations on top
of a hydrostatic equilibrium, the numerical resolution has to be increased to
the point that the truncation errors do not obscure these small perturbations.
This may result in prohibitively high computational costs, especially in
several space dimensions.

A design principle to overcome the challenge was introduced by Greenberg and
Leroux \cite{Greenberg1996} leading to the concept of so-called well-balanced
schemes.
In these schemes, a discrete equivalent of the equilibrium is exactly
satisfied.
Therefore, they possess the ability to maintain discrete equilibrium states
down to machine precision and are capable of resolving small equilibrium
perturbations effectively.
Many well-balanced schemes have been designed, especially for the shallow
water equations with non-trivial bottom topography, see e.g.
\cite{LeVeque1998,Audusse2004,Noelle2006} and references therein.
An extensive review on well-balanced schemes for many different applications is
also given in the book by Gosse \cite{Gosse2013}.

Well-balanced schemes for the Euler equations with gravitation have received a
considerable amount of attention in the recent literature.
First, LeVeque and Bale \cite{LeVeque1999} have applied the quasi-steady
wave-propagation algorithm \cite{LeVeque1998} to the Euler equations with
gravity.
Few years later, Botta et al. \cite{Botta2004539} designed a well-balanced
finite volume scheme for numerical weather prediction applications.
More recently, several well-balanced finite volume
\cite{LeVeque2011,Kaeppeli2014199,Desveaux2014,Chandrashekar2015,2016A&A...587A..94K,ToumaKoleyKlingenberg2016,LiXing2016,Kaeppeli2017,ChertockCuiKurganovEtAl2018}
, finite difference \cite{Xing2013,LiXing2018a}
and discontinuous Galerkin \cite{Li2015,ChandrashekarZenk2017,LiXing2018}
schemes have been presented.
Magnetohydrostatic steady state preserving well-balanced finite volume schemes
were devised in \cite{Fuchs20104033}.
To the best of our knowledge, many of the mentioned schemes are at most
second-order accurate and only \cite{Xing2013,Li2015,ChandrashekarZenk2017,LiXing2018a,LiXing2018} go to higher orders.
However, with the notable exception of \cite{ChandrashekarZenk2017}, it appears
that these schemes need the equilibrium to be predetermined.

In fact, equation \cref{eq:equilibrium} only specifies a mechanical equilibrium.
In order to fully characterize the equilibrium a thermal variable, such as
the specific entropy $s$ or the temperature $T$, needs to be supplemented.
As a concrete astrophysically relevant example of a stationary state we
consider the case of constant entropy.
The relevant thermodynamic relation for isentropic hydrostatic equilibrium is
\begin{equation}
  \label{eq:intro_0020}
  \ud h = T \ud s + \frac{\ud p}{\rho},
\end{equation}
where $h$ is the specific enthalpy
\begin{equation}
  \label{eq:intro_0030}
  h = e + \frac{p}{\rho}
  ,
\end{equation}
$T$ the temperature and $s$ the specific entropy.
Then we can write \cref{eq:equilibrium} for the isentropic case
($\ud s = 0$) as
\begin{equation}
  \label{eq:intro_0040}
  \frac{1}{\rho} \nabla p = \nabla h = - \nabla \phi
  .
\end{equation}
The last equation can then be trivially integrated to obtain
\begin{equation}
  \label{eq:intro_0050}
  h + \phi = const.
\end{equation}
In \cite{Kaeppeli2014199} this equilibrium was used to build a second-order
accurate well-balanced finite volume scheme.
Along the same lines, well-balanced schemes for isothermal hydrostatic
equilibrium can be constructed \cite{Kaeppeli2017}.
In the latter case, the relevant thermodynamic potential is the Gibbs free
energy.

In this paper, we extend the well-balanced finite volume schemes
\cite{Kaeppeli2014199} beyond second-order accuracy.
The scheme possesses the following novel features:
\begin{itemize}
\item An arbitrarily high-order accurate local hydrostatic profile is
      constructed based on the equilibrium \eqref{eq:intro_0050}.
\item An arbitrarly high-order equilibrium preserving reconstruction is
      designed on the basis of any standard high-order reconstruction
      procedure.
\item A well-balanced source term discretization is built from the equilibrium
      preserving reconstruction.
\item It is well-balanced for any consistent numerical flux, which allows a
      straightforward implementation within any standard finite volume method.
\item It is well-balanced for multi-dimensional hydrostatic equilibria.
\item It is not tied to any particular equation of state such as the
      ideal gas law.
      This is important, especially for astrophysical applications.
\end{itemize}

The rest of the paper is structured as follows: the well-balanced finite volume
scheme is presented in section \ref{sec:nm}.
Extensive numerical results are presented in section \ref{sec:numex} and
conclusions are provided in section \ref{sec:conc}.


\section{Numerical Method}
\label{sec:nm}
\subsection{One-dimensional scheme}
\label{sec:nm_1d}
We first consider the Euler equations with gravitation
\cref{eq:euler_continuity,eq:euler_momentum,eq:euler_energy} in one space
dimension and write them in the following compact form
\begin{equation}
  \label{eq:nm_1d_0010}
    \frac{\partial \bu}{\partial t}
  + \frac{\partial \bbf}{\partial x} = \bbs
\end{equation}
with
\begin{equation}
  \label{eq:nm_1d_0020}
  \bu = \begin{bmatrix}
          \rho     \\
          \rho v_x \\
          E
        \end{bmatrix}
  \; , \quad
  \bbf = \begin{bmatrix}
           \rho v_x       \\
           \rho v_x^2 + p \\
           (E + p) v_x
         \end{bmatrix}
  \quad \mathrm{and} \quad
  \bbs = - \begin{bmatrix}
             0        \\
             \rho     \\
             \rho v_x
           \end{bmatrix} \frac{\partial \phi}{\partial x}
  ,
\end{equation}
where $\bu$, $\bbf$ and $\bbs$ are the vectors of conserved variables, fluxes
and source terms.
An equation of state (EoS) $p = p(\rho,e)$ relates the pressure to the
density $\rho$ and specific internal energy $e$ (or any other thermodynamic
quantity such as specific entropy $s$ or temperature $T$).
For example, a simple EoS is provided by the ideal gas law
\begin{equation}
  \label{eq:nm_1d_0030}
  p = \rho e ( \gamma - 1)
  ,
\end{equation}
where $\gamma$ is the ratio of specific heats.
We stress that the well-balanced scheme derived below is not tied to any
particular form of EoS, which is crucial especially in astrophysical
applications.

In the next section we will briefly describe a standard high-order
finite-volume discretization and it's core components in order to fix the
notation.
The following sections will then describe our novel well-balanced scheme in
detail.

\subsubsection{Finite-volume discretization}
\label{subsubsec:nm_1d_fv}
For the numerical approximation of \eqref{eq:nm_1d_0010}, the spatial domain
of interest is discretized by a number of cells or finite volumes
$I_{i} = [x_{i-1/2},x_{i+1/2}]$.
Here $x_{i\pm1/2}$ denotes the left and right cell interface, respectively,
and $x_{i} = (x_{i-1/2} + x_{i+1/2})/2$ the cell center of $I_{i}$.
For ease of presentation, we assume a regular cell size
$\Delta x = x_{i+1/2} - x_{i-1/2}$.
Nevertheless, varying cell sizes can easily be accommodated for.

A one-dimensional semi-discrete finite volume scheme is then given by
\begin{equation}
  \label{eq:nm_1d_fv_0010}
  \frac{\mathrm{d} \mbU_{i}}{\mathrm{d} t}
  = \cL(\mbU)
  = - \frac{1}{\Delta x}
        \left( \bF_{i+1/2} - \bF_{i-1/2} \right)
      + \mbS_{i}
  ,
\end{equation}
where $\mbU_{i} = \mbU_{i}(t)$ denotes the approximate cell average of the
conserved variables in cell $I_{i}$ at time $t$.
It approximates the exact cell average $\mbu_{i} = \mbu_{i}(t)$ of the true
solution $\bu(t,x)$ at time $t$:
\begin{equation}
  \label{eq:nm_1d_fv_0020}
  \mbU_{i}(t) \approx \mbu_{i}(t)
              =       \frac{1}{\Delta x} \int_{I_{i}} \bu(t,x) \byd{x}
  .
\end{equation}
In the following, a quantity with an overbar indicates a cell average while a
quantity without indicates a point value.
By $\mbS_{i}(t)$ is denoted the approximate cell average of the true
source terms at time $t$:
\begin{equation}
  \label{eq:nm_1d_fv_0030}
  \mbS_{i}(t)
    \approx \mbbs_{i}(t)
    =       \frac{1}{\Delta x}
            \int_{I_{i}} \bbs(\bu,\frac{\partial \phi}{\partial x})
            \byd{x}
  .
\end{equation}
Note that we have suppressed the time dependence of the gravitational potential
since we are mainly concerned with flows close to hydrostatic equilibrium and
for ease of notation.

\paragraph{Numerical flux}
The numerical flux is obtained by solving (approximately) the Riemann problem
at cell interfaces
\begin{equation}
  \label{eq:nm_1d_fv_0040}
  \bF_{i+1/2} = \cF ( \bU_{i+1/2-} ,\bU_{i+1/2+} )
  ,
\end{equation}
where the point values $\bU_{i+1/2\mp}$ are the cell interface extrapolated
conserved variables and $\cF$ is a consistent, i.e.
$\cF(\bu,\bu) = \bbf(\bu)$, and Lipschitz continuous numerical flux
function.

Below, we will make use of the HLLC approximate Riemann solver with simple
wave speed estimates from \cite{BattenClarkeLambertEtAl1997,Toro1997}.
Though, our well-balanced scheme is independent of this particular choice.

\paragraph{Reconstruction}
The purpose of a reconstruction procedure $\cR$ is to compute accurate point
values of the approximate solution $\bU_{i}(t,x)$ within each cell from the
cell averages $\mbU$.
We denote such a reconstruction procedure, which recovers a $r$-th order
accurate point value of a quantity $c$ at location $x$ within cell $I_{i}$ from
the cell averages $\mmc$, by
\begin{equation}
  \label{eq:nm_1d_fv_0050}
  c_{i}(x) = \cR(x;\{\mmc_{k}\}_{k \in S_{i}})
  .
\end{equation}
Here $S_{i}$ is the stencil for the reconstruction procedure for cell $I_{i}$,
i.e.\ $S_{i}$ is a finite set of neighbors of $I_{i}$.

The values of the conserved variables extrapolated to the interface are then
given by
\begin{equation*}
  \bU_{i+1/2-} = \bU_{i  }(t,x_{i+1/2})
               = \cR\left(x_{i+1/2};\{\mbU_{k}\}_{k \in S_{i  }}\right)
  \quad \text{and} \quad
  \bU_{i+1/2+} = \bU_{i+1}(t,x_{i+1/2})
               = \cR\left(x_{i+1/2};\{\mbU_{k}\}_{k \in S_{i+1}}\right)
  .
\end{equation*}
Many such reconstruction procedures have been developed and a non-exhaustive
list includes the Total Variation Diminishing (TVD) methods (see e.g.
\cite{Leer1979,Harten1983}), the Piecewise-Parabolic Method (PPM)
\cite{Colella1984}, Essentially Non-Oscillatory (ENO) (see e.g.
\cite{HartenEngquistOsherEtAl1987}), Weighted ENO (WENO) (see e.g.
\cite{Shu2009} and references therein) and Central WENO (CWENO) methods
(see e.g. \cite{CraveroPuppoSempliceEtAl2017} and references therein).

In the scheme derived below we will use a CWENO type reconstruction procedure.
This choice is motivated by the fact that CWENO provides an entire
reconstruction polynomial defined everywhere in a cell, which is convenient for
the evaluation of the gravitational source terms.
However, our scheme is independent of this particular choice.

\paragraph{Source term discretization}
The approximate cell average of the source term $\mbS_{i}$ is obtained by
numerical integration.
Let $\cQ_{i}$ denote a $q$-th order accurate quadrature rule over cell $I_{i}$.
Then the cell average of the source term is approximated by
\begin{equation}
  \label{eq:nm_1d_fv_0060}
  \mbS_{i} = \frac{1}{\Delta x}
             \cQ_{i} \left(\bbs(\bU,\frac{\partial \phi}{\partial x}) \right)
           = \frac{1}{\Delta x}
             \sum_{\alpha=1}^{N_{q}}
                 \omega_{\alpha}
               ~ \bbs \left( \bU_{i}(t,x_{i,\alpha})
                            ,\frac{\partial \phi}{\partial x}(x_{i,\alpha})
                      \right)
  ,
\end{equation}
where the $x_{i,\alpha} \in I_{i}$ and $\omega_{\alpha}$ denote the $N_{q}$
quadrature nodes and weights of $Q_{i}$, respectively.
For example, the two-point Gauss-Legendre quadrature rule can be used, which is
the choice we will make below.
The point values of the conserved variables at the quadrature nodes
$\bU_{i}(t,x_{i,\alpha})$ are obtained by the reconstruction procedure:
\begin{equation}
  \label{eq:nm_1d_fv_0070}
  \bU_{i}(t,x_{i,\alpha}) = \cR
                            \left( x_{i,\alpha};\{\mbU_{k}\}_{k \in S_{i}}
                            \right)
  .
\end{equation}
If the gravitational potential is known analytically, it can be evaluated
directly at the quadrature nodes.
If it is not, then a suitable interpolation has to be applied.

\paragraph{Temporal discretization}
The temporal domain of interest $[0,T]$ is discretized into time steps
$\Delta t = t^{n+1} - t^{n}$, where the superscript $n$ labels the
different time levels.
For the temporal integration, the high-order strong stability-preserving
Runge-Kutta (SSP-RK) schemes \cite{GottliebShuTadmor2001} can be used.
In particular, we use the third-order SSP-RK method for the numerical
results presented in this paper
\begin{equation}
  \label{eq:SSP-RK3}
  \begin{aligned}
    \mbU_{i}^{(1)} & = \mbU_{i}^{n} + \Delta t \cL(\mbU^{n}) \\
    \mbU_{i}^{(2)} & = \frac{3}{4} \mbU_{i}^{n}
                     + \frac{1}{4} \left(  \mbU_{i}^{(1)}
                                        + \Delta t \cL(\mbU^{(1)})
                                  \right) \\
    \mbU_{i}^{n+1} & = \frac{1}{3} \mbU_{i}^{n}
                     + \frac{2}{3} \left(  \mbU_{i}^{(2)}
                                         + \Delta t \cL(\mbU^{(2)})
                                  \right)
    ,
  \end{aligned}
\end{equation}
where $\cL$ denotes the spatial discretization operator from
\eqref{eq:nm_1d_fv_0010}.
Furthermore, the time step $\Delta t$ has to fulfill a certain CFL condition.

This concludes the description of a standard high-order finite volume scheme
for the Euler equations.
We refer to the excellent books available in the literature for detailed
derivations, e.g. \cite{Godlewski1996,Hirsch2007,Laney1998,LeVeque1998}.
However, a standard reconstruction procedure and source term discretization
will in general not preserve a discrete equivalent of hydrostatic equilibrium.
In order to achieve this, we need the ingredients presented in the following
two sections \ref{subsubsec:nm_1d_wbrec} and \ref{subsubsec:nm_1d_wbsrc}.

\subsubsection{Local hydrostatic reconstruction}
\label{subsubsec:nm_1d_wbrec}
The local hydrostatic reconstruction consists of two parts.
First, within each cell a high-order accurate equilibrium profile that is
consistent with the cell-averaged conserved variables is determined.
Second, the cell's equilibrium profile is extrapolated to neighboring cells
to perform a high-order accurate reconstruction of the equilibrium
perturbation.

We begin by describing how the local high-order accurate equilibrium profile is
determined.
Within the $i$-th cell $I_{i}$, we define a subcell equilibrium reconstruction
of the specific enthalpy $h_{eq,i}(x)$ by assuming \eqref{eq:intro_0050} as
\begin{equation}
  \label{eq:nm_1d_wbrec_0010}
  h_{eq,i}(x) = h_{0,i} + \phi_{i} - \phi(x)
  .
\end{equation}
Here $h_{0,i} = h_{eq,i}(x_{i})$ and $\phi_{i} = \phi(x_{i})$ are point values
of the specific enthalpy and the gravitational potential at the cell center,
respectively.
In the following, we assume that the gravitational potential can be evaluated
anywhere, either because it is a given function or obtained by a suitable
interpolation.

In combination with the (assumed constant) equilibrium entropy $s_{0,i}$ in
cell $I_{i}$, the equilibrium density $\rho_{eq,i}(x)$ and internal energy
density $\rho e_{eq,i}(x)$ profiles can be computed through the EoS:
\begin{equation*}
  \rho_{eq,i}(x) = \rho(h_{eq,i}(x),s_{0,i})
  \quad \text{and} \quad
  \rho e_{eq,i}(x) = \rho e(h_{eq,i}(x),s_{0,i})
  .
\end{equation*}
The computational complexity of this computation depends strongly on the
functional form of the EoS.
For the ideal gas case, explicit expressions are given in \ref{sec:appendix01}.

We note that the equilibrium specific enthalpy $h_{0,i}$ and entropy $s_{0,i}$
are not specified so far.
In order to fix $h_{0,i}$ and $s_{0,i}$, we demand that the equilibrium density
and internal energy density profiles agree up to the desired order of accuracy
with their respective cell average in cell $I_{i}$.
Hence, we seek $h_{0,i}$ and $s_{0,i}$ such that
\begin{equation}
  \label{eq:nm_1d_wbrec_0020}
  \begin{aligned}
  \mrho_{i}  & = \frac{1}{\Delta x} \cQ_{i}(\phantom{e}\rho_{eq,i})
               = \frac{1}{\Delta x}
                 \sum_{\alpha=1}^{N_{q}}  \omega_{\alpha}
                                        ~ \rho  (h_{eq,i}(x_{i,\alpha})
                                                ,s_{0,i})               \\
  \mrhoe_{i} & = \frac{1}{\Delta x} \cQ_{i}(\rho e_{eq,i})
               = \frac{1}{\Delta x}
                 \sum_{\alpha=1}^{N_{q}}  \omega_{\alpha}
                                        ~ \rho e(h_{eq,i}(x_{i,\alpha})
                                                ,s_{0,i})
  ,
  \end{aligned}
\end{equation}
where $\cQ_{i}$ denotes the previously introduced $q$-th order accurate
quadrature rule over cell $I_{i}$.
In the above expression, an estimate of the cell average of the internal energy
density $\mrhoe_{i}$ is needed.
We simply estimate it directly from the cell-averaged conserved variables by
\begin{equation}
  \label{eq:nm_1d_wbrec_0030}
  \mrhoe_{i} = \mE_{i} - \frac{1}{2} \frac{\mrhov_{x,i}^{2}}{\mrho_{i}}
  ,
\end{equation}
which is exact at equilibrium ($v_x \equiv 0$).

Note that, in general, \eqref{eq:nm_1d_wbrec_0020} represents a nonlinear system
of two equations in the equilibrium specific enthalpy at cell center $h_{0,i}$
and the (constant) specific entropy $s_{0,i}$.
This system must be solved iteratively, e.g. with Newton's method.
In practice, the iterative process is started from the specific entropy and
enthalpy computed from the cell-averaged conserved variables $\mbU_{i}$.
The cost of this iterative process is mitigated by the fact that it is local to
each cell and the initial guess is a spatially second order accurate estimate,
i.e.\ a very small two-by-two system of equations must be solved,
independently, in every cell starting from a good initial guess.
For the ideal gas law, the system can be reduced to a single nonlinear equation
for which existence and uniqueness of the solution can be guaranteed under very
weak requirements.
This is shown in \ref{sec:appendix01}.

Once $h_{0,i}$ and $s_{0,i}$ have been fixed, we have the following high-order
accurate representation of the equilibrium in cell $I_{i}$:
\begin{equation}
  \label{eq:nm_1d_wbrec_0040}
  \bU_{eq,i}(x) = \begin{bmatrix}
                    \rho_{eq,i}(x)   \\
                    0                \\
                    \rho e_{eq,i}(x)
                  \end{bmatrix}
  .
\end{equation}

Next we develop the high-order equilibrium preserving reconstruction procedure.
The idea is to decompose the solution into an equilibrium and a (possibly
large) perturbation part.
Within cell $I_{i}$, the equilibrium part is simply given by the previously
derived equilibrium profile $\bU_{eq,i}(x)$.
The perturbation part is obtained by applying the standard reconstruction $\cR$
procedure on the equilibrium perturbation cell averages
\begin{equation}
  \label{eq:nm_1d_wbrec_0050}
  \delta \bU_{i}(x) = \cR
                      \left(
                        x; \{\mbU_{k} - \cQ_{k}(\bU_{eq,i})\}_{k \in S_{i}}
                      \right)
  ,
\end{equation}
which results in a $\min(q,r)$-th order accurate representation of the
equilibrium perturbation in cell $I_{i}$.
Note that the equilibrium perturbation cell average in cell $I_{k}$ is obtained
by taking the difference between the actual cell average $\mbU_{k}$ in cell
$I_{k}$ and the cell average of the equilibrium profile $\bU_{eq,i}(x)$ in cell
$I_{k}$.
The latter is evaluated by applying the $I_{k}$ cell's quadrature rule
$\cQ_{k}$ to $\bU_{eq,i}(x)$.

The full equilibrium preserving reconstruction $\cW$ is then obtained by simply
adding the equilibrium profile to the perturbation
\begin{equation}
  \label{eq:nm_1d_wbrec_0060}
  \bU_{i}(x) = \cW(x;\{\mbU_{k}\}_{k \in S_{i}})
             = \bU_{eq,i}(x) + \delta \bU_{i}(x)
  .
\end{equation}
We observe that, by construction, this reconstruction will preserve any
equilibrium of the form \eqref{eq:intro_0050}, since the perturbation $\delta
\bU_{i}(x)$ vanishes under these conditions.

\begin{remark}
Any function can be written as some other function plus the difference. Clearly,
this difference can be reconstructed from the cell-averages of the difference.
Therefore, the well-balanced reconstruction procedure \cref{eq:nm_1d_wbrec_0060}
is high-order accurate, for any smooth function $\bU_{eq,i}(x)$. In particular,
the choice of an only second order accurate estimate of $\mrhoe_i$ does not
affect the overall order of the reconstruction.
\end{remark}

\subsubsection{Well-balanced source term discretization}
\label{subsubsec:nm_1d_wbsrc}
For the momentum source discretization, we use the previous splitting of the
cell $I_{i}$'s density $\rho_{i}(x)$ into equilibrium $\rho_{eq,i}(x)$ and
perturbation $\delta\rho_{i}(x)$ as
\begin{equation*}
  \begin{aligned}
  S_{\rho v,i}(x) & = - \rho_{i}(x) \frac{\partial \phi}{\partial x}(x)
                    = - \left( \rho_{eq,i}(x) + \delta\rho_{i}(x) \right)
                      \frac{\partial \phi}{\partial x}(x)
                    = - \rho_{eq,i}(x) \frac{\partial \phi}{\partial x}(x)
                      - \delta\rho_{i}(x) \frac{\partial \phi}{\partial x}(x)
  ,
  \end{aligned}
\end{equation*}
which is clearly a pointwise $\min(q,r)$-th order accurate approximation of the
true source term.
However, a straightforward numerical integration will not result in a
well-balanced scheme.
Instead, we use the fact that for the equilibrium profiles we have
\begin{equation*}
  \frac{\partial p_{eq,i}}{\partial x} = - \rho_{eq,i}
                                           \frac{\partial \phi}{\partial x}
\end{equation*}
by construction.
As a result, the equilibrium part of the momentum source term can be trivially
integrated and numerical integration is only applied to the perturbation part:
\begin{equation}
  \label{eq:nm_1d_wbsrc_0010}
  \mS_{\rho v,i} = \frac{p_{eq,i}(x_{i+1/2}) - p_{eq,i}(x_{i-1/2})}{\Delta x}
                 - \frac{1}{\Delta x}
                   \cQ_{i} \left(
                             \delta\rho_{i} \frac{\partial \phi}{\partial x}
                           \right)
  .
\end{equation}
Since we are only concerned with stationary equilibria, the energy source term
$\mS_{E,i}$ discretization is left unchanged from \eqref{eq:nm_1d_fv_0060}.

We summarize the developed high-order well-balanced finite volume scheme in
the following theorem:
\begin{theorem}
\label{thm:nm_1d}
Consider the scheme \eqref{eq:nm_1d_fv_0010} with a consistent and Lipschitz
continuous numerical flux $\cF$, a $r$-th order accurate spatial reconstruction
procedure $\cR$, a $q$-th order accurate quadrature rule $\cQ$, the hydrostatic
reconstruction $\cW$ \cref{eq:nm_1d_wbrec_0060} and the gravitational source
term $\cS$ \cref{eq:nm_1d_fv_0060} (with \cref{eq:nm_1d_wbsrc_0010}).

This scheme has the following properties:
\begin{enumerate}[(i)]
\item The scheme is consistent with \eqref{eq:nm_1d_0010} and it is
      $\min(q,r)$-th order accurate in space (for smooth solutions).
\item The scheme is well-balanced and preserves the discrete hydrostatic
      equilibrium given by \eqref{eq:intro_0050} and $v_x = 0$
      exactly.
\end{enumerate}
\end{theorem}
\begin{proof}
(i) The consistency and formal order of accuracy of the scheme is
straightforward.

(ii) Let the hydrostatic equilibrium \eqref{eq:intro_0050} be characterized by
the constant specific entropy $s$ and specific enthalpy profile $h_{eq}(x)$.
The equilibrium conserved variables are then given by
$u_{eq}(x) = [\rho(h_{eq}(x), s),0,\rho e(h_{eq}(x),s]^T$
and let $\bU_i(0) = \frac{1}{\Delta x} \cQ_i\left( \bu_{eq}\right)$ be the
discrete initial conditions.
Then the iterative process for solving \cref{eq:nm_1d_wbrec_0020} will, in
each cell, find the local equilibrium $h_{0,i} = h_{eq}(x_i)$ and
$s_{0,i} = s$.
We prove this fact for ideal gases in \ref{sec:appendix01}.
Therefore, in every cell $\delta \bU_i(x) = \cR(x; \{0\}_{k\in S_i}) = 0$.
Hence, we have $\bU_{i+1/2-} = \bU_{i+1/2+}$ and by consistency of the
numerical flux $\bF_{i+1/2} = \bbf(\bU_{i+1/2-}) = [0,p_{eq}(x_{i+1/2}),0]^T$.
Likewise, by definition \cref{eq:nm_1d_wbsrc_0010} the cell-averaged source
term becomes
$\mbS_i = \frac{1}{\Delta x} [0,p_{eq}(x_{i+1/2}) - p_{eq}(x_{i-1/2}),0]^T$.
By plugging the above expressions for the  numerical flux and source term into
the semi-discrete finite volume scheme \cref{eq:nm_1d_fv_0010} we get
\begin{equation*}
  \frac{\mathrm{d} \mbU_{i}}{\mathrm{d} t}
  = \cL(\mbU)
  = - \frac{1}{\Delta x}
        \left( \bF_{i+1/2} - \bF_{i-1/2} \right)
      + \mbS_{i}
  = 0
\end{equation*}
Thus the scheme is well-balanced as claimed.
\end{proof}

\begin{remark}
The presented scheme reduces to the second-order accurate scheme presented in
\cite{Kaeppeli2014199} by setting the quadrature rule $\cQ$ to the midpoint
rule and the reconstruction procedure $\cR$ to piecewise linear.
\end{remark}

\subsection{Extension to several space dimensions}
\label{subsec:nm_md}
We now describe the extension of our well-balanced scheme for hydrostatic
equilibrium to the multi-dimensional case.
For ease of presentation, we describe it for two dimensions and the extension
to three dimensions is straightforward.
As in the one-dimensional case, we briefly introduce a standard high-order
finite volume scheme and then detail the well-balanced scheme.

The two-dimensional Euler equations with gravity in Cartesian coordinates are
given by
\begin{equation}
  \label{eq:nm_md_0010}
    \frac{\partial \bu }{\partial t}
  + \frac{\partial \bbf}{\partial x}
  + \frac{\partial \bg }{\partial y} = \bbs
\end{equation}
with
\begin{equation}
  \label{eq:nm_md_0020}
  \bu = \begin{bmatrix}
          \rho     \\
          \rho v_x \\
          \rho v_y \\
          E
        \end{bmatrix}
  \; , \quad
  \bbf = \begin{bmatrix}
           \rho v_x       \\
           \rho v_x^2 + p \\
           \rho v_y v_x   \\
           (E + p) v_x
         \end{bmatrix}
  \; , \quad
  \bg = \begin{bmatrix}
          \rho v_y       \\
          \rho v_x v_y   \\
          \rho v_y^2 + p \\
          (E + p) v_y
        \end{bmatrix}
  \quad \mathrm{and} \quad
  \bbs = \bbs_{x} + \bbs_{y}
       = \begin{bmatrix}
             0        \\
           - \rho     \\
             0        \\
           - \rho v_x
         \end{bmatrix} \frac{\partial \phi}{\partial x}
       + \begin{bmatrix}
             0        \\
             0        \\
           - \rho     \\
           - \rho v_y
         \end{bmatrix} \frac{\partial \phi}{\partial y}
  ,
\end{equation}
where $\bu$ is the vector of conserved variables, $\bbf$ and $\bg$ the fluxes
in $x$- and $y$-direction, and $\bbs$ the gravitational source terms.

We consider a rectangular spatial domain
$\Omega = [x_{\min},x_{\max}] \times [y_{\min},y_{\max}]$ discretized
uniformly (for ease of presentation) by $N_{x}$ and $N_{y}$ cells or finite
volumes in $x$- and $y$-direction, respectively.
The cells are labeled by $I_{i,j} = I_{i} \times I_{j} = [x_{i-1/2},x_{i+1/2}] \times [y_{j-1/2},y_{j+1/2}]$
and the constant cell sizes by
$\Delta x = x_{i+1/2} - x_{i-1/2}$ and
$\Delta y = y_{j+1/2} - y_{j-1/2}$.
We denote the cell centers by $x_{i} = (x_{i-1/2} + x_{i+1/2})/2$ and
$y_{j} = (y_{j-1/2} + y_{j+1/2})/2$.
Integrals of some quantity $c$ over the cell faces are approximated by $q$-th
order accurate quadrature rules as
\begin{equation}
  \label{eq:nm_md_0021}
  \begin{aligned}
  Q_{i\pm1/2,j}(c)
  & = \sum_{\beta=1}^{N_{q}} \omega_{\beta}   ~ c(x_{i\pm1/2},y_{j,\beta})
  \approx \int_{I_{j}} c(x_{i\pm1/2},y) ~ \mathrm{d}x
\\
  Q_{i,j\pm1/2}(c)
  & = \sum_{\alpha=1}^{N_{q}} \omega_{\alpha} ~ c(x_{i,\alpha},y_{j\pm1/2})
  \approx \int_{I_{i}} c(x,y_{i\pm1/2}) ~ \mathrm{d}y
  ,
  \end{aligned}
\end{equation}
where the $x_{i,\alpha} \in I_{i}$, $y_{j,\beta} \in I_{j}$ and
$\omega_{\alpha}$, $\omega_{\beta}$ denote the $N_{q}$ quadrature nodes and
weights, respectively.
Likewise, integrals over the cells are approximated by
\begin{equation}
  \label{eq:nm_md_0022}
  Q_{i,j}(c)
  = \sum_{\alpha=1}^{N_{q}} \sum_{\beta=1}^{N_{q}}
    \omega_{\alpha} \omega_{\beta} ~ c(x_{i,\alpha},y_{j,\beta})
  \approx \int_{I_{i,j}} c(x,y) ~ \mathrm{d}x ~ \mathrm{d}y
  .
\end{equation}

A semi-discrete finite volume scheme for the numerical approximation of
\eqref{eq:nm_md_0010} then takes the following form
\begin{equation}
  \label{eq:nm_md_0030}
  \frac{\mathrm{d} \mbU_{i,j}}{\mathrm{d} t}
  = \cL(\mbU)
  = - \frac{1}{\Delta x}
      \left( \bF_{i+1/2,j} - \bF_{i-1/2,j} \right)
    - \frac{1}{\Delta y}
      \left( \bG_{i,j+1/2} - \bG_{i,j-1/2} \right)
    + \mbS_{i,j}
  ,
\end{equation}
where $\mbU_{i,j}$ denotes the approximate cell averages of the conserved
variables, $\bF_{i\pm1/2,j}$ and $\bG_{i,j\pm1/2}$ the facial averages of the
fluxes through the cell boundary and $\mbS_{i,j}$ the cell averages of the
source term.
The fluxes are obtained by applying the above quadrature rules along the cell
boundary to the numerical flux formulas $\cF$ and $\cG$ in respective
direction:
\begin{equation}
  \label{eq:nm_md_0040}
  \begin{aligned}
  \bF_{i+1/2,j} & = \frac{1}{\Delta y}
                    Q_{i+1/2,j} \left( \cF(\bU_{i  ,j},\bU_{i+1,j}) \right) \\
  \bG_{i,j+1/2} & = \frac{1}{\Delta x}
                    Q_{i,j+1/2} \left( \cG(\bU_{i,j  },\bU_{i,j+1}) \right)
  ,
  \end{aligned}
\end{equation}
where $\bU_{i,j} = \bU_{i,j}(x, y)$ is a suitable reconstruction to be defined
in detail at a later point.
Similarly, the source term is obtained by quadrature over the cell
\begin{equation}
  \label{eq:nm_md_0050}
  \mbS_{i,j} = \frac{1}{\Delta x \Delta y} \cQ_{i,j}(\bbs(\bU,\nabla\phi))
  .
\end{equation}
In the evaluation of the quadrature rules, a $r$-th reconstruction procedure
$\cR$ is used to obtain pointwise representations of the solution from the
cell-averaged conserved variables:
\begin{equation}
  \label{eq:nm_md_0051}
  \bU_{i,j}(x,y) = \cR
                     \left(
                       x,y; \left\{ \mbU_{k,l} \right\}_{(k,l) \in S_{i,j}}
                     \right)
  .
\end{equation}
Here $S_{i,j}$ is the stencil of the reconstruction for cell $I_{i,j}$.
Many such reconstruction procedures have been developed in the literature
and we refer to the references previously mentioned in section
\ref{subsubsec:nm_1d_fv}.

As in the one-dimensional case, we need two ingredients to construct our
well-balanced scheme.
The first is a high-order equilibrium preserving reconstruction and the second
is a well-balanced discretization of the momentum source terms.

Let us begin with the description of the first ingredient and consider cell
$I_{i,j}$.
Then the high-order equilibrium preserving reconstruction $\cW$ takes the
following form
\begin{equation}
  \label{eq:nm_md_0060}
  \bU_{i,j}(x,y) = \cW
                   \left(x,y; \left\{\mbU_{k,l}\right\}_{(k,l) \in S_{i,j}}
                   \right)
                 = \bU_{eq,i,j}(x,y) + \delta \bU_{i,j}(x,y)
  ,
\end{equation}
which again separates the solution into an equilibrium $\bU_{eq,i,j}$ and a
(possibly large) perturbation $\delta \bU_{i,j}$.

The equilibrium profile is built from \eqref{eq:intro_0050}, which is indeed
also valid in more than one dimensions.
Hence, we construct the local equilibrium profile in cell $I_{i,j}$ by
\begin{equation}
  \label{eq:nm_md_0070}
  h_{eq,i,j}(x,y) = h_{0,i,j} + \phi_{i,j} - \phi(x,y)
  ,
\end{equation}
where $h_{0,i,j} = h_{eq,i,j}(x_{i},y_{j})$ and
$\phi_{i,j} = \phi(x_{i},y_{j})$ are the point values of the specific enthalpy
and the gravitational potential at cell center, respectively.
Given the (constant) equilibrium entropy $s_{0,i,j}$, the equilibrium profiles
of density $\rho_{eq,i,j}$ and internal energy density $\rho e_{eq,i,j}$ can be
computed through the EoS.

The equilibrium enthalpy at cell center $h_{0,i,j}$ and the (constant) entropy
$s_{0,i,j}$ are again fixed by demanding agreement with the local cell averages
up to the desired order of accuracy:
\begin{equation}
  \label{eq:nm_md_0080}
  \begin{aligned}
  \mrho_{i,j}  & = \frac{1}{\Delta x \Delta y} \cQ_{i,j}(\rho_{eq,i,j})   \\
  \mrhoe_{i,j} & = \frac{1}{\Delta x \Delta y} \cQ_{i,j}(\rho e_{eq,i,j})
  .
  \end{aligned}
\end{equation}
Here $\mrhoe_{i,j}$ is the cell average of the internal energy density,
which we estimate simply from the cell-averaged conserved variables by
\begin{equation}
  \label{eq:nm_md_0090}
  \mrhoe_{i,j} = \mE_{i,j}
               - \frac{1}{2 \mrho_{i,j}}
                 \left(\mrhov_{x,i,j}^{2} + \mrhov_{y,i,j}^{2} \right)
  ,
\end{equation}
The latter estimate is again exact at equilibrium.
As in the one-dimensional case, these equations represent, in general, a
nonlinear system of two equations in the equilibrium specific enthalpy at cell
center $h_{0,i,j}$ and the (constant) specific entropy $s_{0,i,j}$.
Their resolution proceeds as in the one-dimensional case.
In the end, we have the following equilibrium profile
\begin{equation}
  \label{eq:nm_md_0100}
  \bU_{eq,i,j}(x,y) = \begin{bmatrix}
                        \rho_{eq,i,j}(x,y)   \\
                        0                    \\
                        0                    \\
                        \rho e_{eq,i,j}(x,y)
                      \end{bmatrix}
  .
\end{equation}

The perturbation part is reconstructed as in the one-dimensional case by
\begin{equation}
  \label{eq:nm_md_0110}
  \delta \bU_{i,j}(x,y) = \cR
                          \left(
                            x,y; \left\{  \mbU_{k,l}
                                        - Q_{k,l}(\bU_{eq,i,j})
                                 \right\}_{(k,l) \in S_{i,j}}
                          \right)
  .
\end{equation}
This simply extrapolates the cell's local equilibrium profile, computes
equilibrium cell averages by numerical integration, and uses the standard
reconstruction procedure to obtain a high-order representation of the
perturbation.

We observe that the reconstruction procedure \eqref{eq:nm_md_0060} preserves
the equilibrium by construction, since $\delta \bU_{i,j}$ vanishes, and it is
$\min(q,r)$-th order accurate in and away from equilibrium (for sufficiently
smooth solutions).

Like in the one-dimensional case, only the momentum source terms need to be
modified.
The well-balanced momentum source terms are simply obtained on a
dimension-by-dimension basis
\begin{equation}
  \label{eq:nm_md_0120}
  \begin{aligned}
  \mS_{\rho v_{x},i,j}
  & = \frac{1}{\Delta x}
      \left( Q_{i+1/2,j}(p_{eq,i,j}) - Q_{i-1/2,j}(p_{eq,i,j}) \right)
    - \frac{1}{\Delta x \Delta y} \cQ_{i,j}
      \left( \delta\rho_{i,j} \frac{\partial \phi}{\partial x} \right) \\
  \mS_{\rho v_{y},i,j}
  & = \frac{1}{\Delta y}
      \left( Q_{i,j+1/2}(p_{eq,i,j}) - Q_{i,j-1/2}(p_{eq,i,j}) \right)
    - \frac{1}{\Delta x \Delta y} \cQ_{i,j}
      \left( \delta\rho_{i,j} \frac{\partial \phi}{\partial y} \right)
  .
  \end{aligned}
\end{equation}
This completes the description of the two-dimensional well-balanced scheme
for hydrostatic equilibrium and its properties are summarized in the
corollary below:
\begin{corollary}
\label{cor:nm_md}
Consider the scheme \eqref{eq:nm_md_0030} with consistent and Lipschitz
continuous numerical fluxes $\cF$ and $\cG$, a $r$-th order accurate spatial
reconstruction procedure $\cR$, a $q$-th order accurate quadrature rule $\cQ$,
the hydrostatic reconstruction $\cW$ \eqref{eq:nm_md_0060} and the
gravitational source term $\cS$ \eqref{eq:nm_md_0050} (with
\eqref{eq:nm_md_0120}).

This scheme has the following properties:
\begin{enumerate}[(i)]
\item The scheme is consistent with \eqref{eq:nm_md_0010} and it is
      $\min(q,r)$-th order accurate in space (for smooth solutions).
\item The scheme is well-balanced and preserves the discrete hydrostatic
      equilibrium given by \eqref{eq:intro_0050} and $v_x = v_y = 0$
      exactly.
\end{enumerate}
\end{corollary}
\begin{proof}
The proof follows directly by applying theorem \ref{thm:nm_1d}
dimension-by-dimension.
\end{proof}


\section{Numerical Experiments}
\label{sec:numex}
In this section we assess the performance of our well-balanced scheme on a
series of numerical experiments. For comparison, we also present results
obtained with a standard (unbalanced) base scheme.
The fully-discrete finite volume base scheme consists of
\begin{itemize}
\item the temporally third-order accurate SSP-RK scheme for time integration
      (see \cite{GottliebShuTadmor2001}),
\item the spatially  third-order accurate CWENO3 \cite{Levy2000} reconstruction
      procedure $\cR$,
\item the spatially fourth-order accurate two-point Gauss-Legendre quadrature
      rule for $\cQ$.
\end{itemize}
Overall the scheme is third-order accurate in space and time. This scheme is
conditionally stable under the usual CFL condition. We use a CFL number of
$C_{\mathrm{CFL}} = 0.85$. In the following, we will refer to this scheme as the
unbalanced scheme. The well-balanced scheme is built with the same base
components, but uses the well-balanced reconstruction procedure and source term
computation as outlined in the previous section.

Below, all the initial conditions will be given in functional form
$\bu_0(\bx)$.
The discrete initial conditions are obtained simply by quadrature,
i.e.\
\begin{align}
  \mbU^0_i = \cQ_i(\bu_0), \quad
  \mbU^0_{i,j} = \cQ_{i,j}(\bu_0)
\end{align}
in the one- and two-dimensional case, respectively.
It is important to notice that the well-balancing only requires that the
initial conditions are obtained by the exact same quadrature rule used in the
numerical scheme.
Therefore, showing that the initial conditions are well-balanced will
immediately imply that the preserved discrete state is a high-order accurate
approximation of the exact equilibrium, simply because the discrete initial
conditions are nothing else than a high-order quadrature of the exact
equilibrium.


We will be using three distinct notions of ``error''.
The first error is the usual $L^1$-error
\begin{align}
  err_1(q) := \sum_{i} \Delta x ~ \abs{ \mq_i - \mq_{ref, i} }
  ,
\end{align}
where $q$ is any scalar variable of interest, e.g. $q = \rho, p, v, \dots$.
Furthermore, $\mq_{ref,i}$ is computed by down-sampling a high-resolution
reference solution or, where available, a highly accurate quadrature of an
analytic solution. A subtlety is that even in a well-balanced scheme the
$err_{1}$ of a discrete preserved state is not, in general, zero.

To answer the question of how big the error of a perturbation $\delta q$ from
equilibrium is, we define the $L^1$-error of $\delta q$ as
\begin{align}
  err_1(\delta q)
 := \sum_{i} \Delta x ~ \abs{ (\mq_i - \cQ_{i}(q_{eq})) - \delta \mq_{ref,i} }
\end{align}
where $q_{eq}$ is the background equilibrium profile and $\delta \mq_{ref,i}$ is
the cell-average of the perturbation in a reference solution. This measures the
error of the perturbation from numerical equilibrium. This is subtly different
than the error of the perturbation from the exact equilibrium. The difference is
that $err_1(\delta q)$ conveniently uses the quadrature rule used in the finite
volume method to compute the average of the equilibrium profile, i.e.
$\cQ_{i,j} q_{eq}$. Therefore, for equilibria, a well-balanced scheme should
have zero $err_1(\delta q)$, but may have non-zero $err_1(q)$.

When computing the $err_{1}(\delta q)$ for hydrostatic equilibria, the reference
solution $\cQ_{i,j}(q_{eq})$ is known exactly, it's simply the initial
condition. Therefore, $err_1(\delta q)$ can be computed at a greatly reduced
computational cost by
\begin{align}
  err_{eq,1}(q) := \sum_{i} \Delta x \abs{\mq_i - \cQ_{i}(q_{eq})}
\end{align}
In order to be clear about how the errors where computed we will make the
distinction throughout the numerical experiments.
Moreover, the above error measures readily generalize to the two-dimensional
case.

To characterize a time scale on which a model reacts to perturbations of its
equilibrium, we define the sound crossing time $\tau_{\mathrm{sound}}$
\begin{equation}
  \tau_{\mathrm{sound}} = 2 \int c_s^{-1}\mathrm{d}x
  ,
\end{equation}
where $c_s$ denotes the speed of sound and the integral has to be taken over the
extent of the stationary state of interest. The sound crossing time is basically
the time in which a sound wave travels back and forth through the equilibrium.

We begin by several simple one- and two-dimensional numerical experiments
employing the ideal gas EoS.
The interested reader may readily reproduce these experiments in order to
check his or her implementation.
Finally, we demonstrate the performance of the scheme on a problem involving
a complex multiphysics EoS.

\subsection{One-dimensional Tests}
\label{sec:one_dim_tests}
We consider an isentropic hydrostatic atmosphere in a constant gravitational
field.
The gravitational potential is a linear function $\phi(x) = gx$ where $g$ is
the constant gravitational acceleration.
The initial density and pressure profiles are then given by
\begin{equation}
  \begin{aligned}
    \rho_0(x) = \left(\frac{1}{K} \frac{\gamma-1}{\gamma}
                      (h_{0} - gx)\right)^{1/\gamma-1}
    ,\quad
    p_0(x) = K \rho_0(x)^\gamma
           + A \exp\left(- \frac{(x - 1/2)^2}{0.05^2}\right).
  \end{aligned}
\end{equation}
with the constants $g = 3.15$, $\gamma = 1.4$, $h_0 = 3.75$ and $K = 1$.
The atmosphere's pressure is perturbed by a Gaussian bump of amplitude $A$.
The velocity is set to zero everywhere.

The computational domain is set to $[0,L]$ with $L = 1$ and uniformly
discretized by $N$ cells, i.e.\ we set the cell size $\Delta x = L/N$, the cell
interfaces $x_{i+1/2} = i \Delta x$ and the cell centers
$x_{i} = (x_{i-1/2} + x_{i+1/2})/2$ for $i=1,...,N$.
The following resolutions are used $N = 32,64,128,256,512,1024$.

The boundary conditions are treated as follows.
We extrapolate the local equilibrium from the last physical cell into the left
and right ghost cells by
\begin{equation}
  \begin{aligned}
  \mbU_{i} & = Q_{i}(\bU_{eq,1}) \quad \text{for} \quad i < 1 \\
  \mbU_{i} & = Q_{i}(\bU_{eq,N}) \quad \text{for} \quad i > N
  .
  \end{aligned}
\end{equation}

\subsubsection{Well-balanced property}
\label{subsubsec:one_dim_tests_wb}
We first verify the well-balanced property of our scheme.
For this we evolve the isentropic atmosphere in hydrostatic equilibrium without
pressure perturbation, $A = 0$, up to time $t = 10$. This corresponds to roughly
$6$ sound crossing time ($\tau_{\mathrm{sound}} = 1.6$). The numerical errors
for the density at final time are shown in \cref{tab:one_dim_eq}. The table
clearly shows that the well-balanced scheme maintains the discrete stationary
state to machine precision. Since the initial conditions are the two-point
Gauss-Legendre quadrature of the exact equilibrium, this furthermore shows that
the preserved state is a fourth order accurate approximation of the exact
equilibrium. The unbalanced scheme produces large errors and is unable to
maintain the hydrostatic equilibrium accurately.


\begin{table}[hbt]
  \begin{center}
    \begin{tabular}{r
                S[table-format=3.2e2]r
                S[table-format=3.2e2]r
}
\toprule
N & 
\multicolumn{2}{c}{\cwenonaive} & 
\multicolumn{2}{c}{\cwenoisentropic} \\
 & 
\multicolumn{1}{c}{$err_{eq,1}(\rho)$} & 
\multicolumn{1}{c}{rate} & 
\multicolumn{1}{c}{$err_{eq,1}(\rho)$} & 
\multicolumn{1}{c}{rate} \\
\midrule
  32 &  8.73e-04  &      --  &  9.85e-16  &      --   \\
  64 &  1.38e-04  &     2.66 &  3.96e-15  &    -2.01  \\
 128 &  1.21e-05  &     3.51 &  1.83e-15  &     1.11  \\
 256 &  1.21e-06  &     3.33 &  3.18e-15  &    -0.80  \\
 512 &  1.25e-07  &     3.27 &  4.49e-15  &    -0.50  \\
1024 &  1.31e-08  &     3.25 &  8.34e-15  &    -0.89  \\
\bottomrule
\end{tabular}
    \caption{Convergence data for the one-dimensional test
      \cref{subsubsec:one_dim_tests_wb} without perturbation. We show $err_{eq,1}(\rho)$ at $t
      = 10.0$. On the left we show the error for the unbalanced scheme. It converges
      at slightly higher rates than expected. This is likely due to the fact that the
      initial conditions are a fourth order approximation of the exact cell-averages.
      Clearly, the numerical solution isn't stationary and the truncation error has
      accumulated over time. The right hand side column shows the errors for the
      well-balanced scheme. Note that the errors are at the level of round-off. This
      result also implies that the preserved discrete state is a fourth order
      approximation of the exact equilibrium.}
    \label{tab:one_dim_eq}
  \end{center}
\end{table}

\subsubsection{Small pressure perturbation propagation}
\label{subsubsec:one_dim_tests_small}
Next we add a small pressure perturbation to the isentropic atmosphere in order
to examine the schemes ability to propagate small waves.
The amplitude of the pressure perturbation is set to $A = 10^{-7}$, which generates
one smooth wave propagating upwards and one downwards through atmosphere.
As the waves propagate, they are modified by the density and pressure
stratification of the atmosphere.
We evolve the setup until time $t= 0.2$, shortly before the waves reach the
boundaries.

The errors of the density perturbation $err_1(\delta \rho)$ are shown in
\cref{tab:one_dim_smooth}. The density perturbation is the density at the final
time minus the density of the unperturbed atmosphere. These errors were obtained
on the basis of a reference solution computed by the unbalanced scheme with a
high resolution $N = \num{32768}$. We observe that the errors of the
well-balanced scheme are roughly four orders of magnitude smaller than the
unbalanced scheme. The convergence rate of both the unbalanced and well-balanced
reach the expected rate of three. The somewhat irregular convergence rates of
the unbalanced scheme can be explained by the scheme still being (heavily)
pre-asymtotic at the lower resolutions. The slow convergence rate of the
well-balanced scheme is a feature of the well-balanced scheme. Since it was
designed to have very small errors close to equilibrium.

In \cref{fig:one_dimensional_smooth} the profile of the velocity and the
pressure perturbation are shown at the final time for both the unbalanced (blue
crosses) and well-balanced (red circle) schemes. The well-balanced solution is
shown for $N = 64$. Even at this relatively low resolution the well-balanced
scheme resolves the perturbation well. The errors of the unbalanced scheme for
$N=64$ are too small to be shown on the same plot. Instead we plot the solution
of the unbalanced scheme at $N = 256$. Even at this increased resolution the
perturbation is not approximated well and spurious drifts have developed during
this short period of time.

\begin{table}[hbt]
  \begin{center}
    \begin{tabular}{r
                S[table-format=3.2e2]r
                S[table-format=3.2e2]r
}
\toprule
N & 
\multicolumn{2}{c}{\cwenonaive} & 
\multicolumn{2}{c}{\cwenoisentropic} \\
 & 
\multicolumn{1}{c}{$err_{1}(\delta\rho)$} & 
\multicolumn{1}{c}{rate} & 
\multicolumn{1}{c}{$err_{1}(\delta\rho)$} & 
\multicolumn{1}{c}{rate} \\
\midrule
  32 &  5.84e-06  &      --  &  3.02e-09  &      --   \\
  64 &  6.19e-07  &     3.24 &  8.12e-10  &     1.89  \\
 128 &  1.79e-07  &     1.79 &  1.34e-10  &     2.60  \\
 256 &  3.37e-08  &     2.41 &  1.84e-11  &     2.86  \\
 512 &  4.90e-09  &     2.78 &  2.40e-12  &     2.94  \\
1024 &  6.24e-10  &     2.97 &  3.02e-13  &     2.99  \\
\bottomrule
\end{tabular}
    \caption{Convergence data for the one-dimensional test
      \cref{subsubsec:one_dim_tests_small} with the small perturbation $A =
      10^{-7}$. We show the error of the density perturbation $err_1(\delta \rho)$ at
      $t = 0.2$. On the left we show the errors for the unbalanced scheme. The first
      observation is that the errors are large and even exceed the size of the
      perturbation, the second is that the convergence rates are not quite as
      expected. The reason for the second observation is that given the very small
      perturbation we are trying to resolve, the scheme is likely still pre-asymptotic.
      The errors of the perturbation for the well-balanced scheme is given in the
      right hand side column. The overall error is much smaller, and less than the
      amplitude of the perturbation. Furthermore, the scheme converges at the expected
      rate.
    }
    \label{tab:one_dim_smooth}
  \end{center}
\end{table}

\begin{figure}[hbt]
  \begin{center}
    \includegraphics[width=0.45\columnwidth]{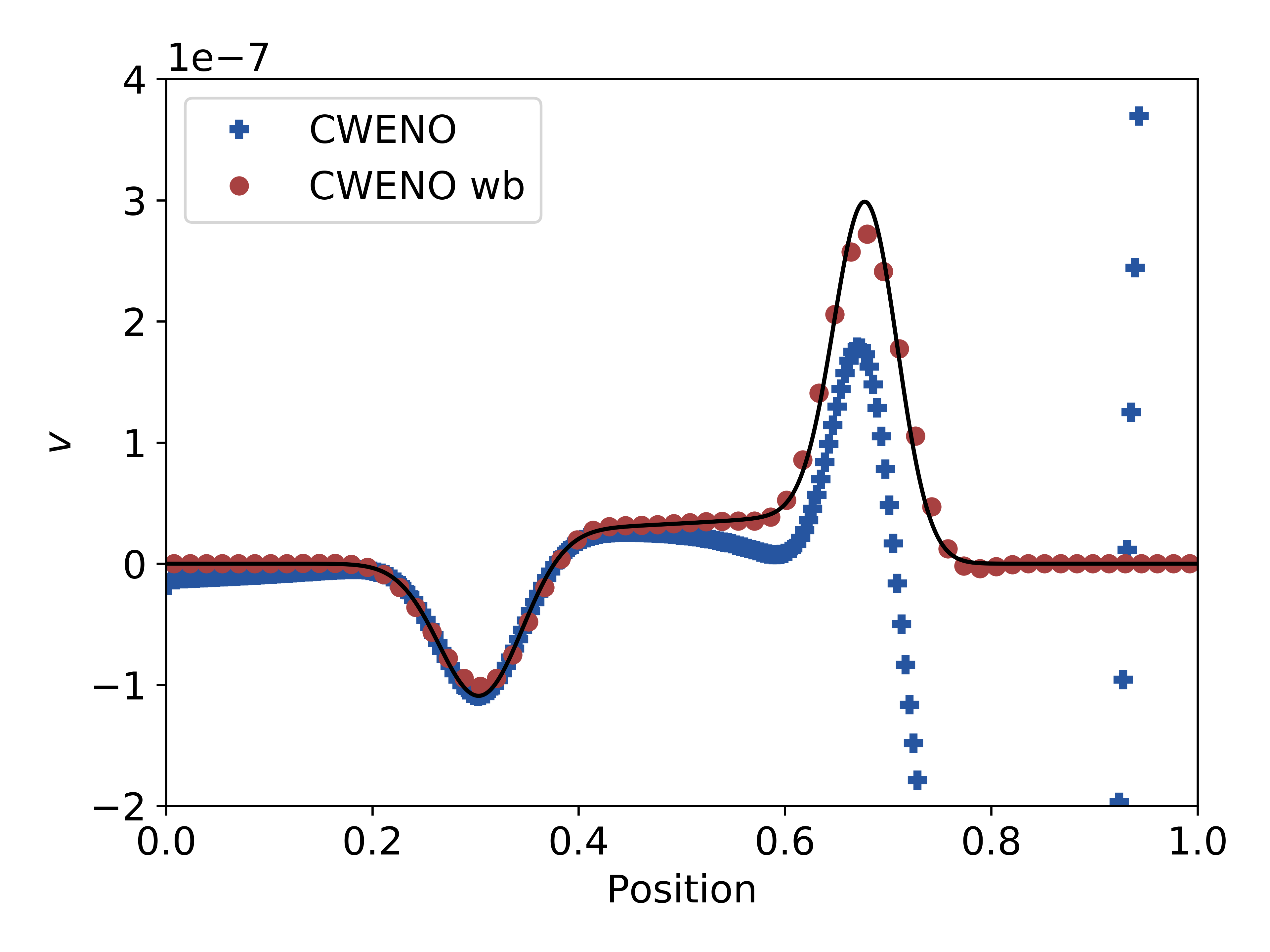}
    \includegraphics[width=0.45\columnwidth]{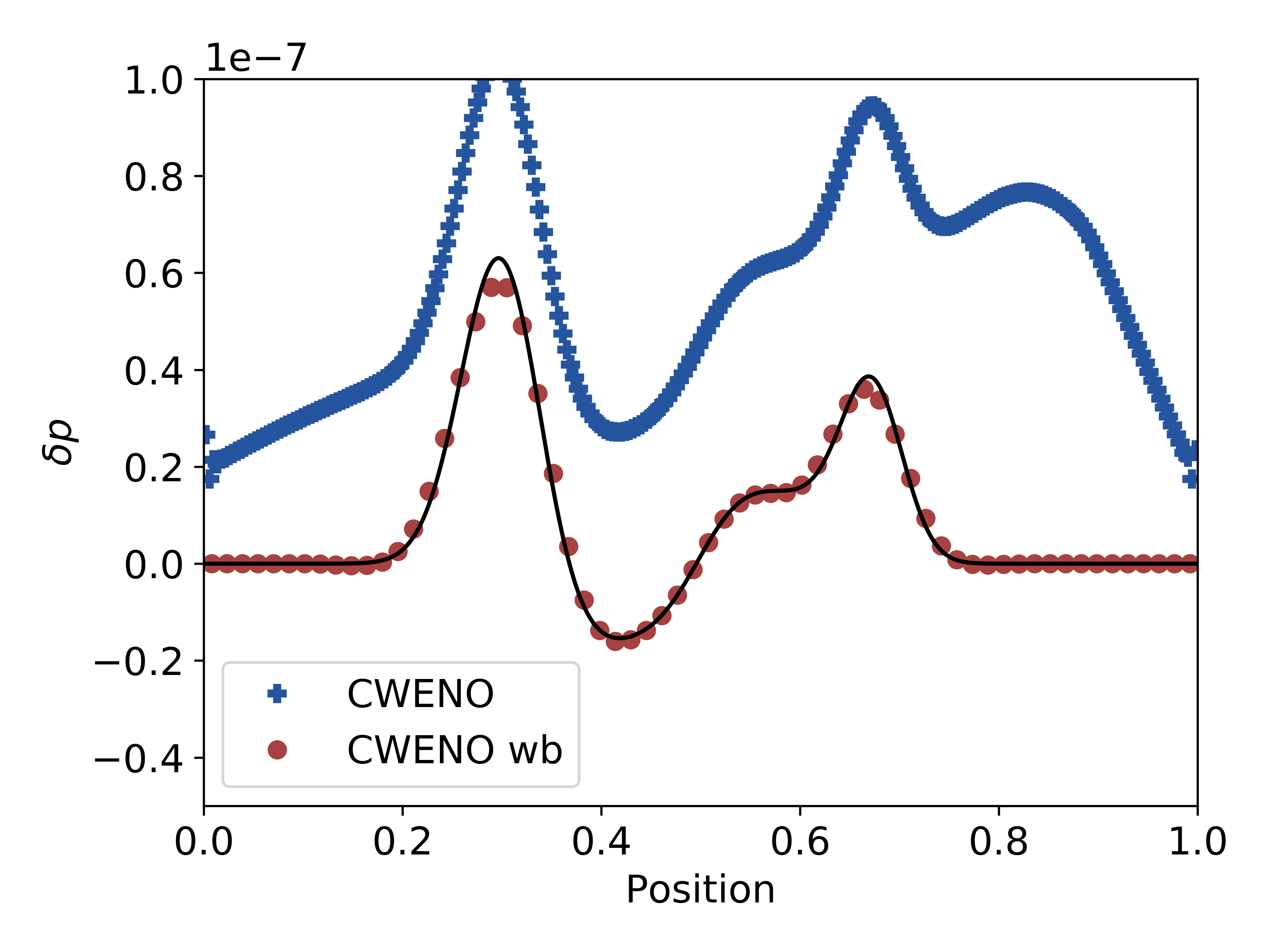}
  \end{center}
  \vspace{-0.2in}
  \caption{Snapshot of the smooth test case, see
\cref{subsubsec:one_dim_tests_small}. The simulation is performed with $N = 256$
and $N=64$ cells for the unbalanced and well-balanced scheme respectively. It is
run up to time $t = 0.2$. On the left we show the velocity, on the right the
pressure perturbation. Even with $N=64$ cells, the well-balanced scheme can
resolve the small perturbation cleanly. The error in the unbalanced scheme on
the other hand is too big to be shown on the plot. The reference solution,
plotted in black, is obtained by a high-resolution $N=\num{32768}$ simulation
using the unbalanced scheme.
}
  \label{fig:one_dimensional_smooth}
\end{figure}

\subsubsection{Large pressure perturbation propagation}
\label{subsubsec:one_dim_tests_large}
For the purpose of testing that the well-balanced reconstruction does not
destroy the robustness of the shock-capturing base scheme, we increase the
pressure perturbation by several orders of magnitude to $A = 10$.
This generates two strong waves quickly steepening into shock waves.
The setup is evolved until time $t = 0.06$.

The plots of the velocity and pressure are shown in
\cref{fig:one_dimensional_blast}.
The two schemes are virtually indistinguishable by eye.
In particular, the well-balanced scheme does not show any oscillations and
performs equally well as the underlying unbalanced scheme.
The well-balancing has not adversely affected the performance of the scheme
away from equilibrium.

\begin{figure}[hbt]
  \begin{center}
    \includegraphics[width=0.45\columnwidth]{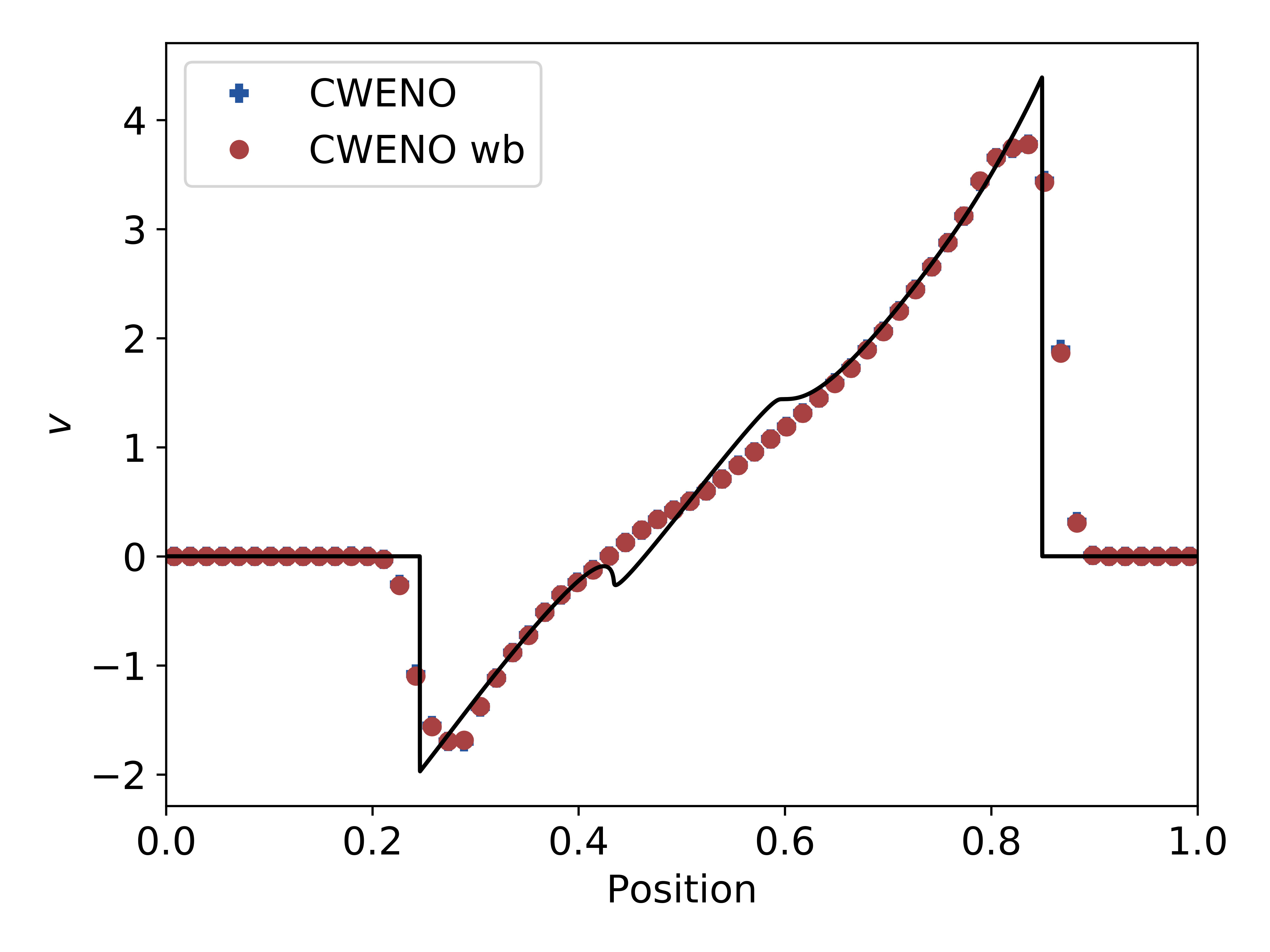}
    \includegraphics[width=0.45\columnwidth]{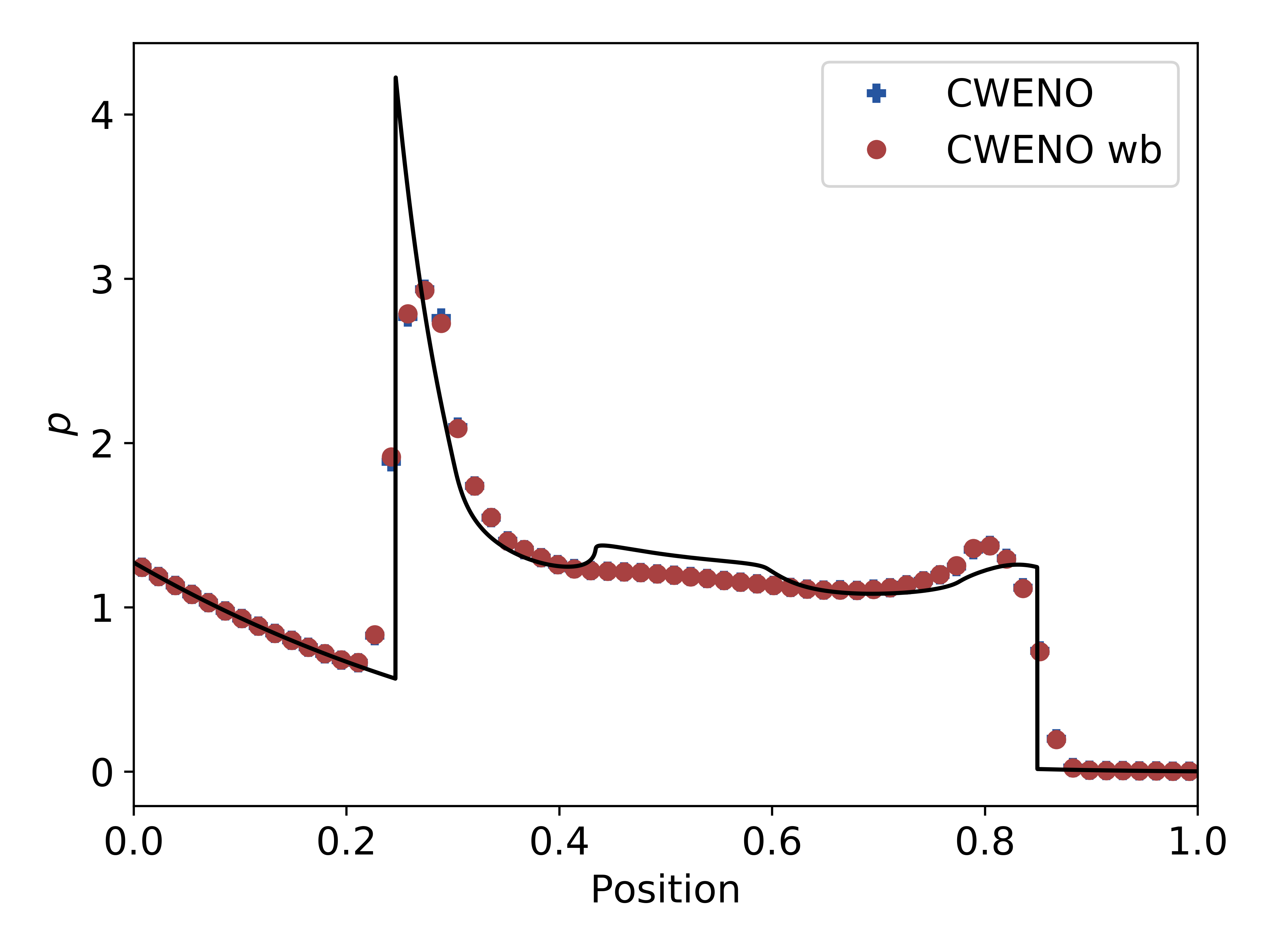}
  \end{center}
  \vspace{-0.2in}
  \caption{Snapshot of one-dimensional test with large perturbation $A = 10$,
see \cref{subsubsec:one_dim_tests_large}. The simulation is performed with $N=64$ cells and
run up to time $t = 0.06$. On the left the velocity is shown, on the right the
pressure. The well-balanced scheme is nearly indistinguishable from the
unbalanced scheme. Clearly, the well-balancing has not affected the quality of
the numerical solution away from equilibrium. Furthermore, no spurious
oscillations are observed. The reference solution, plotted in black, is obtained
by a high-resolution $N=\num{32768}$ simulation using the unbalanced scheme.}
  \label{fig:one_dimensional_blast}
\end{figure}

\clearpage

\subsection{Two-dimensional polytrope}
\label{subsec:numex_2dpoly}
The following numerical experiment is a two-dimensional version of the one
in \cite{Kaeppeli2014199}.
This experiment simulates a so-called polytrope, which is a static
configuration of an adiabatic gaseous sphere held together by self-gravitation.
These model stars are constructed in spherical symmetry from hydrostatic
equilibrium, Poisson's equation and the polytropic relation
$p = K \rho^{\gamma}$, which can be combined into the so-called Lane-Emden
equation (see e.g.\ \cite{1967aits.book.....C}).
The latter equation can be solved analytically for three values of the ratio
of specific heats ($\gamma = 6/5, 2, \infty$).

As in \cite{Kaeppeli2014199} we use $\gamma = 2$.
Then the density and pressure profiles are given by
\begin{align}
  \rho_{0}(r) = \rho_C \frac{\sin(\alpha r)}{\alpha r}, \quad
  p_{0}(r) = K \rho_0(r)^\gamma
\end{align}
where $r$ is the radial coordinate, $\rho_{C}$ is the central density of the
polytrope and
\begin{align}
  \alpha = \sqrt{\frac{4 \pi G}{2K}}
  .
\end{align}
The gravitational potential is given by
\begin{align}
  \phi(r) = - 2 K \rho_C \frac{\sin(\alpha r)}{\alpha r}
  .
\end{align}
In the following we set $K = G = \rho_{C} = 1$.
Note that the polytrope (obviously) fulfills the equilibrium
\cref{eq:intro_0050} for any $r \geq 0$.

We then discretize the problem on the computational domain $[-0.5, 0.5]^2$ by
$N^2$ uniform cells for $N = 32,64,128,256,512,1024$. The conserved variables
are initialized by numerical integration of the conserved variables
$\bu_{0}(x,y) = [\rho_{0}(r),0,0,p_{0}(r)/(\gamma - 1)]^T$ where the radial
coordinate is given by $r^2 = x^2 + y^2$. Note that the velocity is set to zero
in the whole domain.

The boundary conditions are applied along the coordinates axes as in
\cref{sec:one_dim_tests}.
In the corner boundaries (needed by the reconstruction procedure), we
extrapolate the equilibrium from the relevant corner cell in the computational
domain.
For example, the ghost cells in the upper right corner are set as follows
\begin{equation}
  \mbU_{i,j} = Q_{i,j}(\mbU_{eq,N,N})
  \quad \text{for} \quad N < i,j
  .
\end{equation}
The gravitational potential is simply given by the above analytical expression.

\subsubsection{Well-balanced property}
\label{subsubsec:numex_2dpoly_wb}
We begin by evolving the polytrope with the well-balanced and unbalanced
scheme until time $t = 30$ which corresponds to roughly $35$ sound-crossing
times ($\tau_{sound} \approx 0.85$).
The errors are shown in \cref{tab:poly_at_rest}.
The results show that our scheme is well-balanced in two dimensions.
Note that this again implies that the scheme approximates the exact
equilibrium to fourth order in the (usual) $L^1$ norm.
Furthermore, it also works for equilibria which are not grid aligned.
The unbalanced scheme, however, suffers from large spurious deviations.

\begin{table}[hbt]
  \begin{center}
    \begin{tabular}{r
                S[table-format=3.2e2]r
                S[table-format=3.2e2]r
}
\toprule
N & 
\multicolumn{2}{c}{\cwenonaive} & 
\multicolumn{2}{c}{\cwenoisentropic} \\
 & 
\multicolumn{1}{c}{$err_{eq,1}(\rho)$} & 
\multicolumn{1}{c}{rate} & 
\multicolumn{1}{c}{$err_{eq,1}(\rho)$} & 
\multicolumn{1}{c}{rate} \\
\midrule
  32 &  1.57e-03  &      --  &  2.72e-11  &      --   \\
  64 &  1.85e-04  &     3.09 &  4.33e-13  &     5.97  \\
 128 &  2.22e-05  &     3.06 &  5.29e-14  &     3.03  \\
 256 &  2.66e-06  &     3.06 &  1.06e-13  &    -1.00  \\
 512 &  3.07e-07  &     3.12 &  2.04e-13  &    -0.95  \\
1024 &  3.31e-08  &     3.21 &  3.94e-13  &    -0.95  \\
\bottomrule
\end{tabular}
    \caption{Convergence data for the polytrope at rest, see
      \cref{subsubsec:numex_2dpoly_wb}.
      We show $err_{eq,1}(\rho)$ for the unbalanced scheme (left) and the
      well-balanced scheme (right). The unbalanced scheme converges at the
      expected rate, but even with $1024^2$ cells, it has not reached round off.
      The well-balanced scheme is again shown to be in fact well-balanced. Like
      in the one-dimensional test, this implies that the preserved discrete
      state is fourth order accurate. This shows that the scheme
      also works in two-dimensions, even in cases where the gravity is
      non-constant and not grid-aligned.}
    \label{tab:poly_at_rest}
  \end{center}
\end{table}

\subsubsection{Perturbed polytrope}
\label{subsubsec:numex_2dpoly_waves}
Next we add a perturbation to the equilibrium pressure of the polytrope as
\begin{align}
  p(r) = \left(1 + A \exp(-r^2 / 0.05^2) \right) ~ p_{0}(r)
\end{align}
with three different amplitudes $A = 10^{-8}, 10^{-4}, 10^{-2}, 10$.
The setup is evolved up to time $t = 0.2$ shortly before the excited waves
reach the boundary of the computational domain.

The reference solution was computed with a one-dimensional second-order
accurate finite volume scheme (assuming cylindrical symmetry) and
resolution $N = \num{32768}$.

\begin{figure}[htbp]
  \begin{center}
    \includegraphics[width=0.45\columnwidth]{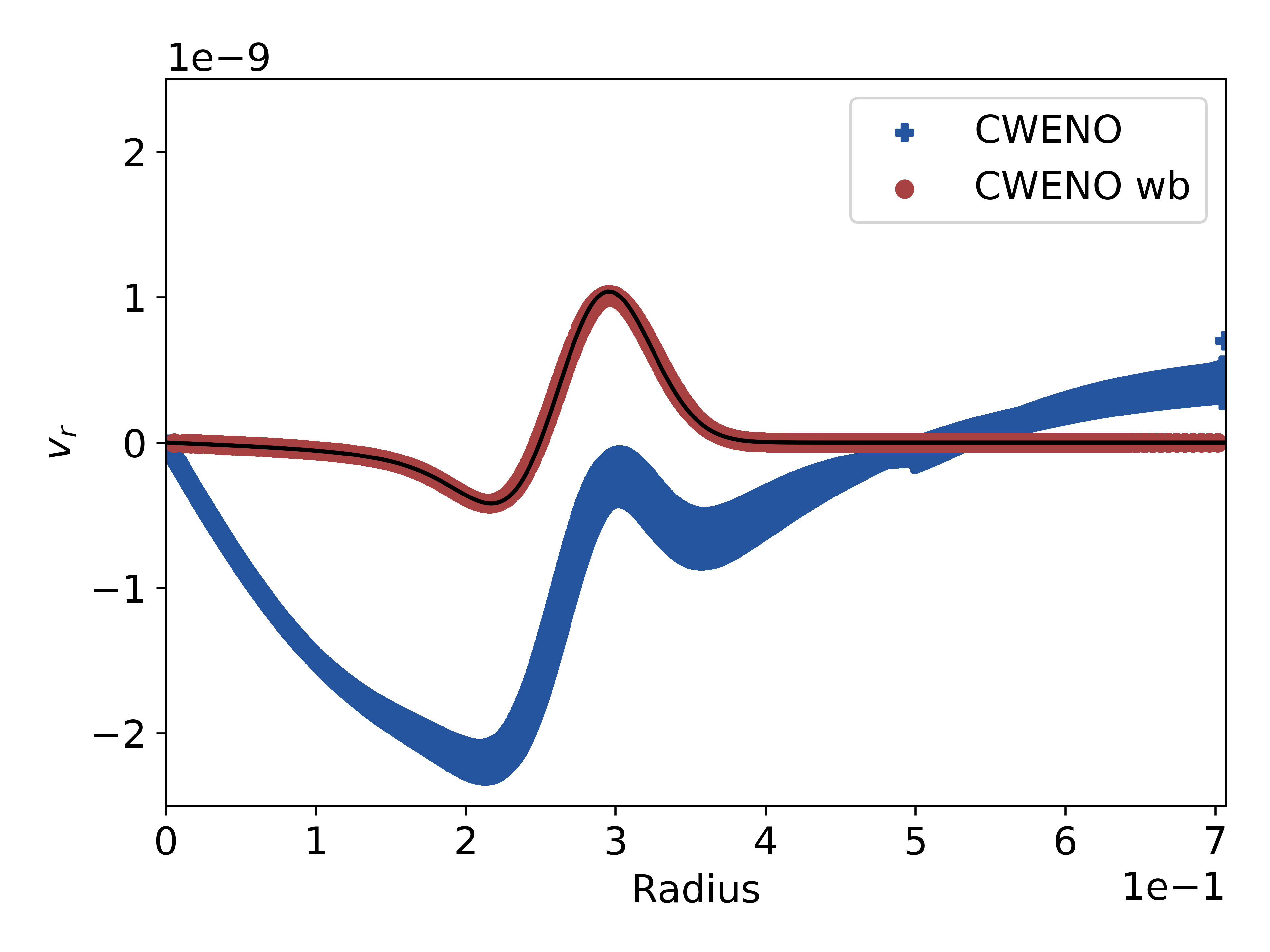}
    \includegraphics[width=0.45\columnwidth]{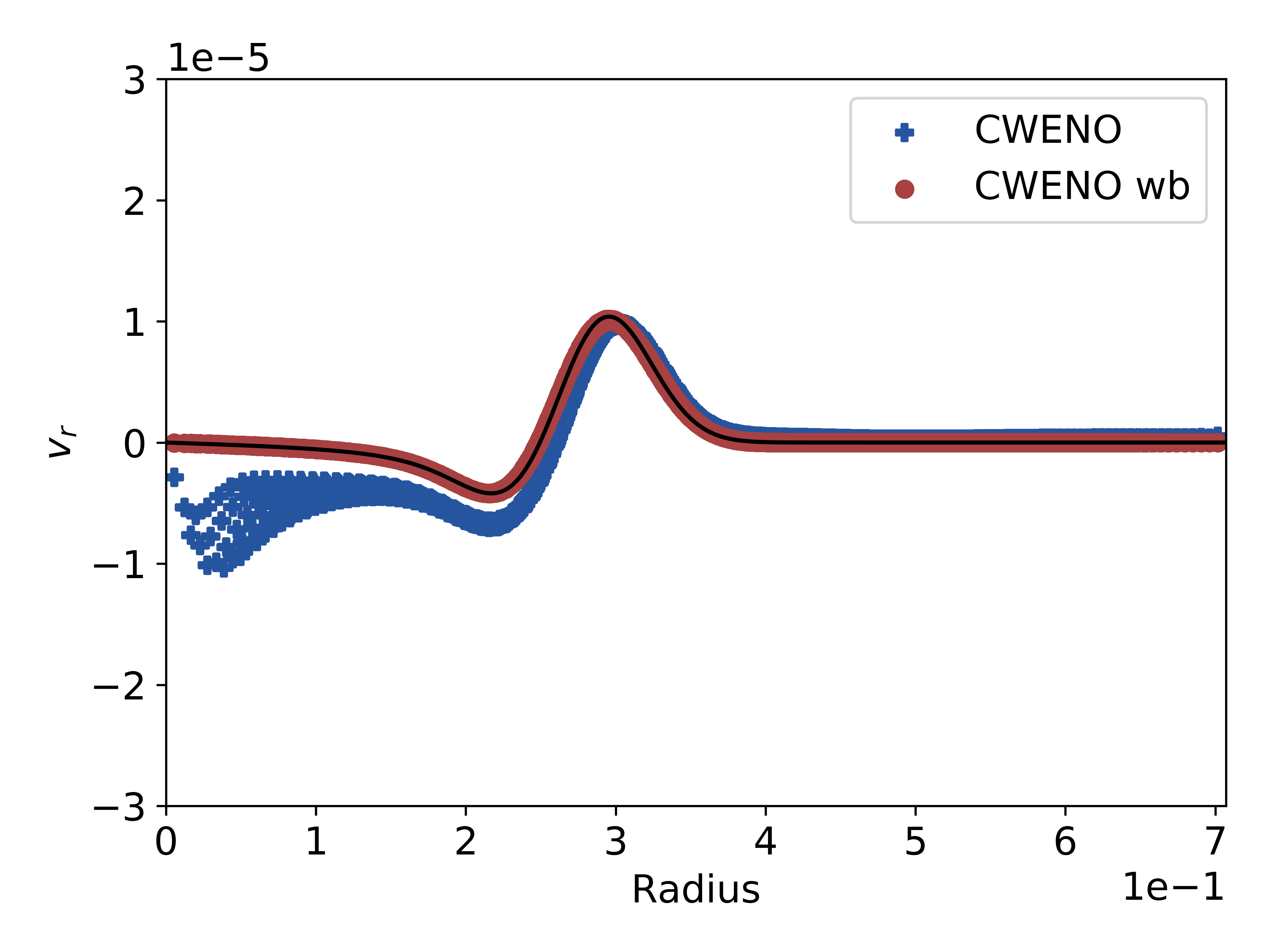}
  \end{center}
  \vspace{-0.2in}
  \caption{Scatter plots of the two-dimensional polytrope, see
    \cref{subsubsec:numex_2dpoly_wb} at time $t = 0.2$. The panel on the left
shows the velocity for $A = 10^{-8}$, the one on the right for $A = 10^{-4}$.
The resolution is generally $N = 128^2$, however, for the smallest perturbation, the
errors of the unbalanced scheme at $N = 128$ exceed the limits of the plot.
Therefore, for $A = 10^{-8}$, we plot the unbalanced scheme at $N = 1024$. In both
cases the well-balanced scheme outperforms the unbalanced scheme. Furthermore,
the well-balanced scheme has no discernible scatter. This is non-trivial since
the radially symmetric solution is approximated on a uniform Cartesian grid
which does not respect the radial symmetry. Furthermore, we see that the
well-balanced scheme, always returns to equilibrium, while the unbalanced scheme
does not. The reference solution was computed with a one-dimensional finite
volume code assuming cylindrical symmetry on $N = \num{32768}$ cells.
  }
  \label{fig:poly_smooth_small}
  \label{fig:poly_smooth_medium}
\end{figure}

\begin{figure}[htbp]
  \begin{center}
    \includegraphics[width=0.45\columnwidth]{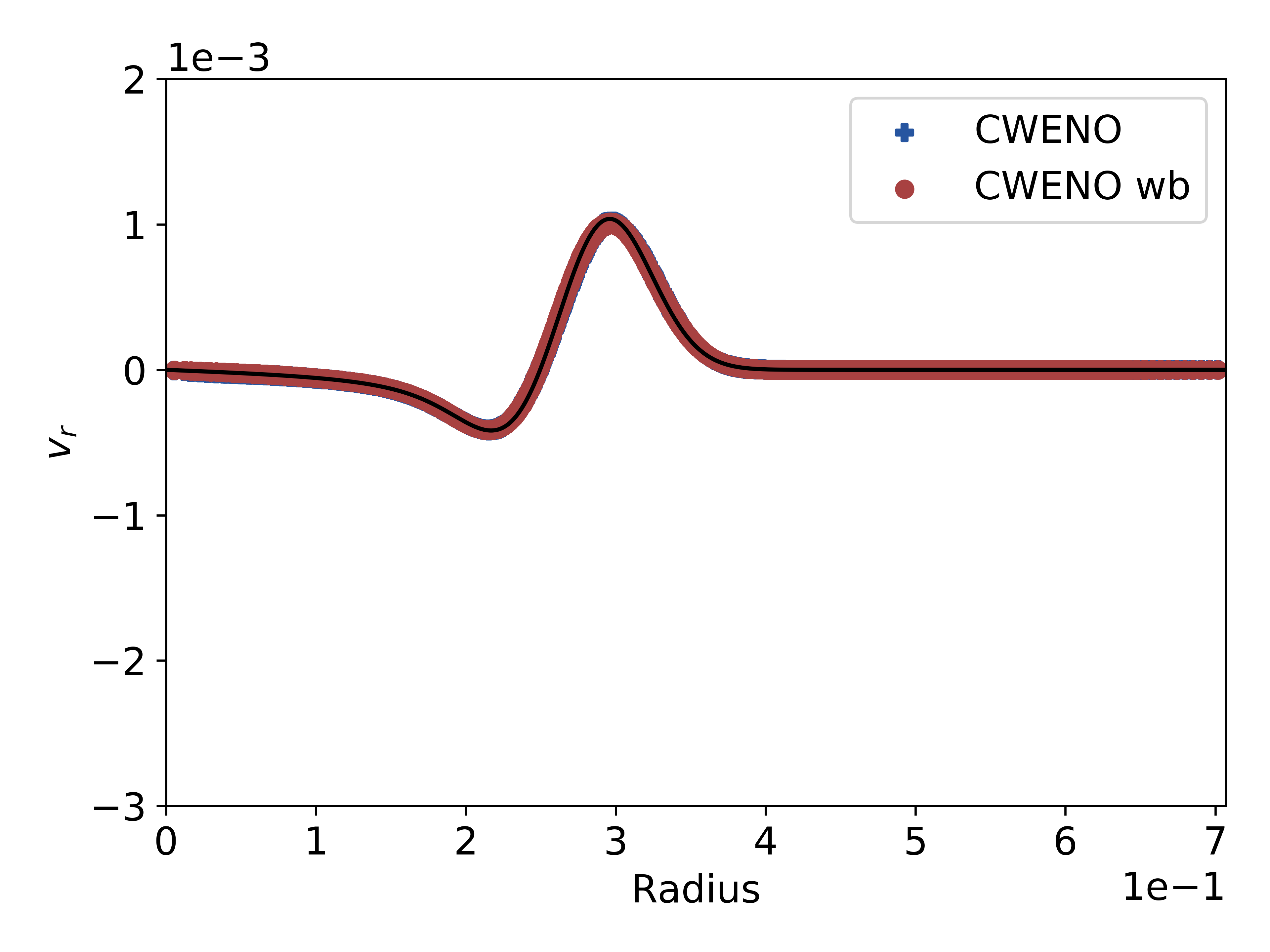}
    \includegraphics[width=0.45\columnwidth]{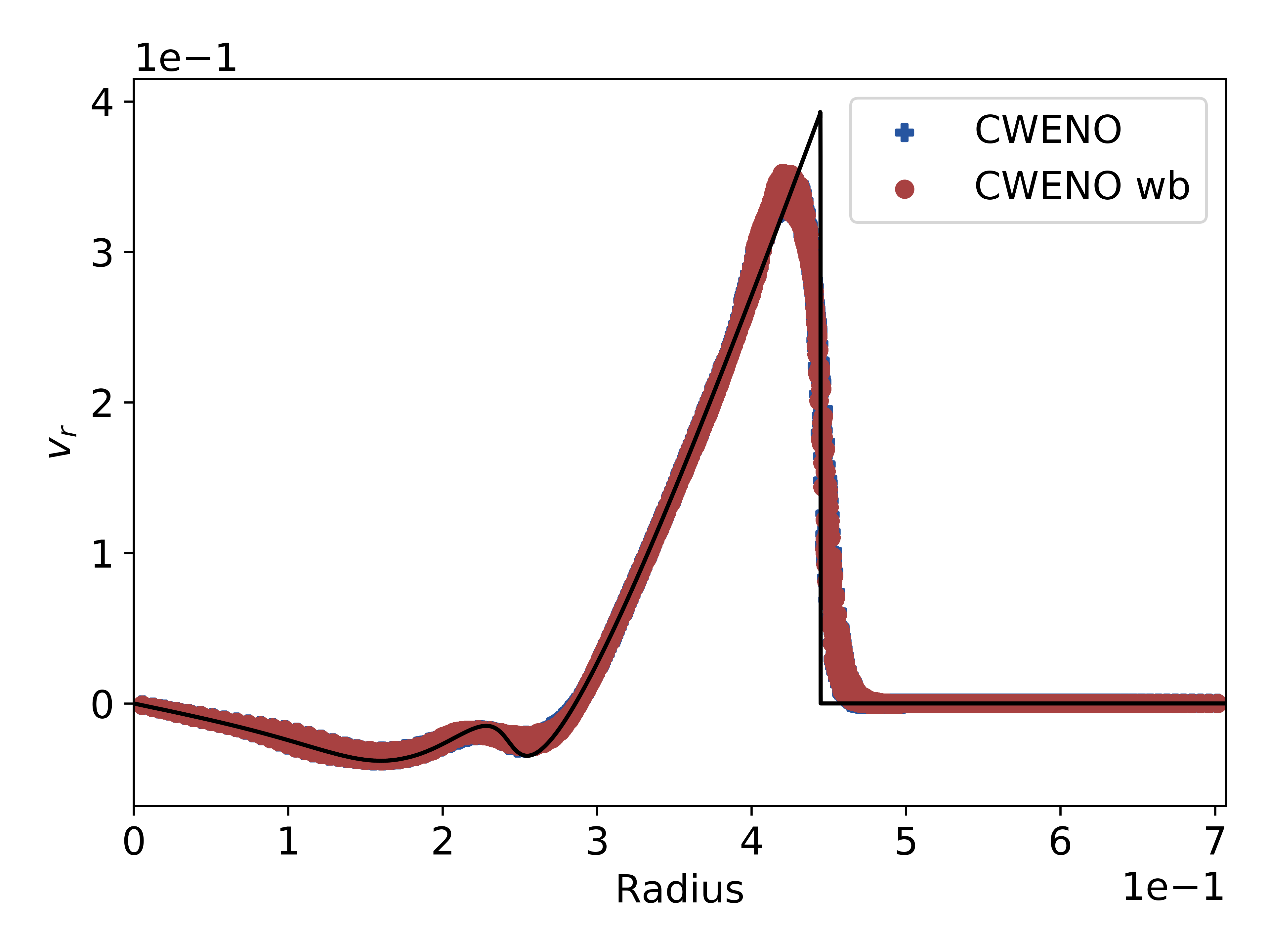}
  \end{center}
  \vspace{-0.2in}
  \caption{Scatter plots of the two-dimensional polytrope, see
    \cref{subsubsec:numex_2dpoly_wb}. The simulation is performed with $N=128^2$
cells and run up to time $t = 0.2$. The panel on the left shows the velocity for
$A = 10^{-2}$, the one on the right for $A = 10$. For these larger perturbations
we see that the unbalanced and well-balanced scheme agree very well. In fact,
\cref{tab:poly_smooth_delta} shows that the errors are on the same order of
magnitude for $A = 10^{-2}$ and differ only about a percent for $A = 10$.
Clearly the well-balancing does not affect the quality of the approximate
solution away from equilibrium. Furthermore, $A = 10$ shows that the
well-balanced scheme resolves discontinuities just as well as the unbalanced scheme.
The reference solution was computed with a one-dimensional finite volume code
assuming cylindrical symmetry on $N = \num{32768}$ cells.
  }
  \label{fig:poly_smooth_large}
  \label{fig:poly_smooth_huge}
\end{figure}

For perturbations of size $A = 10^{-8}$ the
well-balanced scheme clearly outperforms the unbalanced scheme (by at least four
orders of magnitude). Scatter plots of the velocity and pressure perturbation
are shown in \cref{fig:poly_smooth_small}. At $N = 128^2$ the well-balanced
scheme resolves the perturbation well and with no discernible scatter. Which is
not trivial, since the radially symmetric solution is approximated on a uniform
Cartesian grid which does not respect the radial symmetry.

At the next larger perturbation, $A = 10^{-4}$ the well-balanced scheme still
outperforms the unbalanced scheme by a factor $10$. Unlike the unbalanced
scheme, the well-balanced scheme shows no scatter, as can be seen in
\cref{fig:poly_smooth_medium}. Furthermore, once the perturbation has traveled
away from the center of the domain, the solution returns back to equilibrium in
the well-balanced scheme, but not in the unbalanced one. Both schemes converge
at the expected rate.

The second largest perturbation, $A = 10^{-2}$, was chosen such that the
perturbation is very well approximated by the unbalanced scheme, yet small
enough to not develop any discontinuities. The aim is to show that away
from equilibrium the well-balancing does not have a negative impact on the
quality of the solution. This is confirmed in \cref{fig:poly_smooth_large} and \cref{tab:poly_smooth_delta}.

For $A = 10$ both schemes perform equally well, both converge at first order
and the errors differ by approximately one percent. Therefore, well-balancing
has not affected the quality of the solution away from equilibrium. The scatter
plot of the velocity and pressure is shown in \cref{fig:poly_smooth_huge}.
Neither scheme shows any sign of spurious oscillations.

\begin{table}[htbp]
  \begin{center}
    \begin{tabular}{r
                S[table-format=3.2e2]r
                S[table-format=3.2e2]r
}
\toprule
N & 
\multicolumn{2}{c}{\cwenonaive} & 
\multicolumn{2}{c}{\cwenoisentropic} \\
 & 
\multicolumn{1}{c}{$err_{1}(\delta\rho)$} & 
\multicolumn{1}{c}{rate} & 
\multicolumn{1}{c}{$err_{1}(\delta\rho)$} & 
\multicolumn{1}{c}{rate} \\
\midrule
  32 &  5.74e-05  &      --  &  5.01e-11  &      --   \\
  64 &  6.20e-06  &     3.21 &  1.58e-11  &     1.67  \\
 128 &  5.34e-07  &     3.54 &  2.74e-12  &     2.52  \\
 256 &  4.69e-08  &     3.51 &  3.71e-13  &     2.89  \\
 512 &  4.84e-09  &     3.28 &  5.34e-14  &     2.80  \\
1024 &  5.67e-10  &     3.09 &  1.99e-14  &     1.42  \\
\bottomrule
\end{tabular}
    \begin{tabular}{r
                S[table-format=3.2e2]r
                S[table-format=3.2e2]r
}
\toprule
N & 
\multicolumn{2}{c}{\cwenonaive} & 
\multicolumn{2}{c}{\cwenoisentropic} \\
 & 
\multicolumn{1}{c}{$err_{1}(\delta\rho)$} & 
\multicolumn{1}{c}{rate} & 
\multicolumn{1}{c}{$err_{1}(\delta\rho)$} & 
\multicolumn{1}{c}{rate} \\
\midrule
  32 &  5.76e-05  &      --  &  5.00e-07  &      --   \\
  64 &  6.27e-06  &     3.20 &  1.58e-07  &     1.66  \\
 128 &  5.46e-07  &     3.52 &  2.74e-08  &     2.52  \\
 256 &  4.70e-08  &     3.54 &  3.69e-09  &     2.89  \\
 512 &  4.84e-09  &     3.28 &  4.68e-10  &     2.98  \\
1024 &  5.67e-10  &     3.09 &  5.87e-11  &     2.99  \\
\bottomrule
\end{tabular}
    \begin{tabular}{r
                S[table-format=3.2e2]r
                S[table-format=3.2e2]r
}
\toprule
N & 
\multicolumn{2}{c}{\cwenonaive} & 
\multicolumn{2}{c}{\cwenoisentropic} \\
 & 
\multicolumn{1}{c}{$err_{1}(\delta\rho)$} & 
\multicolumn{1}{c}{rate} & 
\multicolumn{1}{c}{$err_{1}(\delta\rho)$} & 
\multicolumn{1}{c}{rate} \\
\midrule
  32 &  8.43e-05  &      --  &  5.50e-05  &      --   \\
  64 &  2.15e-05  &     1.97 &  1.79e-05  &     1.62  \\
 128 &  3.27e-06  &     2.72 &  2.84e-06  &     2.65  \\
 256 &  4.04e-07  &     3.02 &  3.72e-07  &     2.93  \\
 512 &  4.95e-08  &     3.03 &  4.68e-08  &     2.99  \\
1024 &  6.14e-09  &     3.01 &  5.87e-09  &     3.00  \\
\bottomrule
\end{tabular}
    \begin{tabular}{r
                S[table-format=3.2e2]r
                S[table-format=3.2e2]r
}
\toprule
N & 
\multicolumn{2}{c}{\cwenonaive} & 
\multicolumn{2}{c}{\cwenoisentropic} \\
 & 
\multicolumn{1}{c}{$err_{1}(\delta\rho)$} & 
\multicolumn{1}{c}{rate} & 
\multicolumn{1}{c}{$err_{1}(\delta\rho)$} & 
\multicolumn{1}{c}{rate} \\
\midrule
  32 &  2.88e-02  &      --  &  2.98e-02  &      --   \\
  64 &  1.42e-02  &     1.02 &  1.46e-02  &     1.03  \\
 128 &  6.36e-03  &     1.16 &  6.46e-03  &     1.17  \\
 256 &  3.02e-03  &     1.08 &  3.05e-03  &     1.08  \\
 512 &  1.51e-03  &     1.00 &  1.51e-03  &     1.01  \\
1024 &  7.69e-04  &     0.97 &  7.69e-04  &     0.97  \\
\bottomrule
\end{tabular}
    \caption{Convergence data for the polytrope with perturbation. We
      show $err_1(\delta \rho)$ at $t = 0.2$ for the unbalanced (left) and
      well-balanced scheme (right). The first table contains the errors for the
      smallest perturbation $A = 10^{-8}$. Clearly, the well-balanced scheme
      outperforms the unbalanced scheme. Furthermore, the expected rate is
      observed until round off sets in at $N = 512^2$. The second table shows
      the errors for the medium perturbation, $A = 10^{-4}$. The well-balanced
      scheme is still slightly better than the unbalanced one. However the real
      benefit can be seen much more clearly in the scatter plots, c.f.\
      \cref{fig:poly_smooth_medium}. The third table is for $A = 10^{-2}$. We
      clearly see that the well-balancing had no negative affect on the quality of the
      solution, even though the solution is no longer near equilibrium. Finally,
      the fourth table shows the case for $A = 10$. The smooth perturbation turns into a
      discontinuity and only first order convergence can be expected. It's interesting
      to see that the error of the unbalanced and well-balanced scheme differ by only
      about one percent.
}
    \label{tab:poly_smooth_delta}
  \end{center}
\end{table}

\subsubsection{Blast waves}
\label{subsubsec:numex_2dpoly_blast}
In order to further verify that our well-balanced scheme does not deteriorate the
robustness and shock-capturing properties of the base scheme, we add
to the polytrope several localized high pressure regions.
To this end, we add the following pressure perturbation to the equilibrium
polytrope
\begin{align}
  \delta p(\bx) = 100 \sum_{i=1}^6 \mathbf{1}_{B(\bx_i, r)}(\bx)
  ,
\end{align}
where $B_{\bx, R} = \{\bx^\prime \in \IR^2 : \norm{\bx^\prime-\bx} < R\}$
denotes the open ball of radius $R$ centered on $\bx$ and $\mathbf{1}_{B}$ the
indicator function for the set $B$, i.e.
\begin{equation*}
  \mathbf{1}_{B}(\bx) = \begin{cases}
                          1 & \text{if } \bx \in B , \\
                          0 & \text{otherwise.}
                        \end{cases}
\end{equation*}
We setup six ''high pressure balls'' with radii $R = 0.05$ and centers
\begin{align*}
\bx_1 & = [-0.25, 0.3]^T,
\;
\bx_2   = [-0.15, 0.1]^T,
\;
\bx_3   = [0.025, 0.3]^T,
\;
\bx_4   = [0.025, 0.225]^T,
\\
\bx_5 & =
\bx_6   = [0.1, -0.1]^T.
\end{align*}
The velocity is set to zero everywhere.
The initial conditions are shown in the upper panel of \cref{fig:poly_blast}.

We evolve the setup until time $t = 0.02$ with the well-balanced and
unbalanced scheme at resolution $N = 128^2$
We show a snapshot at $t = 0.02$ in \cref{fig:poly_blast}.
Even under these much more extreme conditions with non-trivial wave
interactions, the well-balanced scheme is stable and by eye
indistinguishable from the unbalanced scheme.

\begin{figure}[htbp]
  \begin{center}
    \includegraphics[width=0.45\columnwidth]{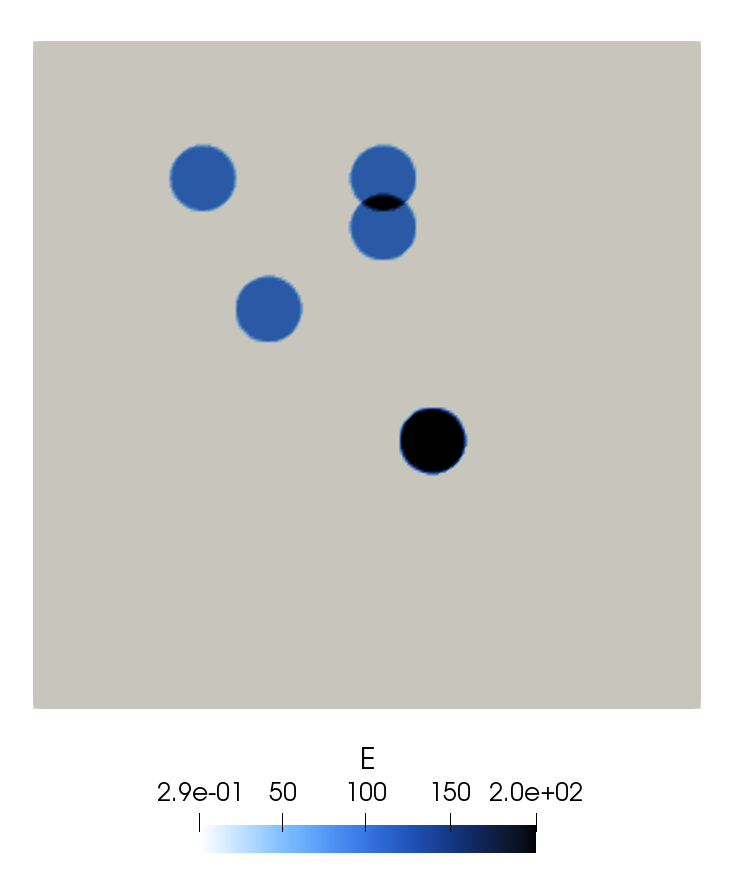}
    \\
    \includegraphics[width=0.45\columnwidth]{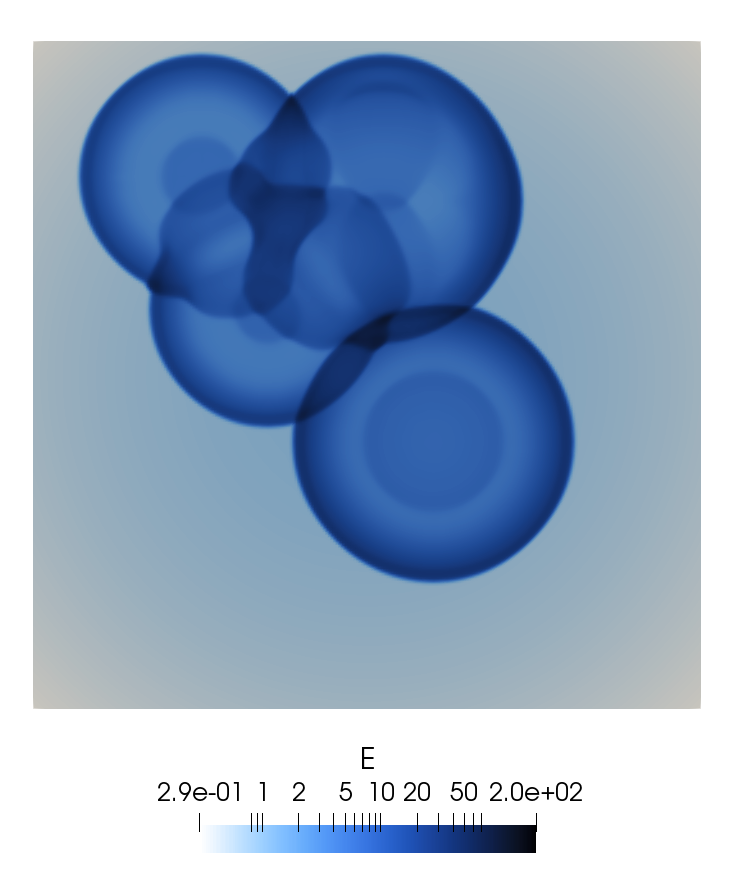}
    \includegraphics[width=0.45\columnwidth]{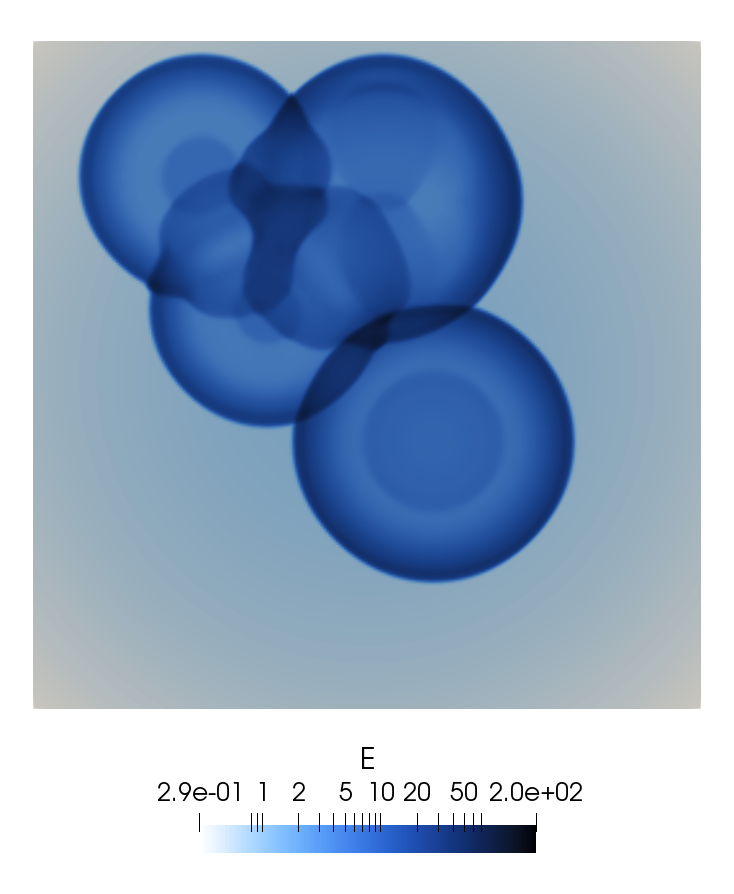}
  \end{center}
  \vspace{-0.2in}
  \caption{Snapshot of the two dimensional blast waves, see
\cref{subsubsec:numex_2dpoly_blast}. The upper image shows the initial total
energy (on a linear scale). The remaining three plots show the total energy at
$t = 0.02$ (on a logarithmic scale) for the unbalanced (bottom, left) and
well-balanced (bottom, right) scheme. Even for these extreme initial conditions
the unbalanced and well-balanced schemes perform equally well.}
  \label{fig:poly_blast}
\end{figure}



\subsection{White Dwarf}
\label{subsec:subsubsec:numex_2dwd}
The final numerical experiment assesses the performance of our well-balanced
scheme on a astrophysically relevant problem involving a complex multiphysics
EoS. We simulate the equilibrium and some perturbations of a model white dwarf.
A white dwarf is the final evolutionary state of a star not massive enough to go
through the final nuclear burning stages and become a neutron star or a black
hole (see e.g.\ \cite{Shapiro2007}).

This numerical experiment is a two-dimensional version of the one presented in
\cite{Kaeppeli2014199}. Likewise, we use the publicly available Helmholtz EoS
(see \cite{Timmes2000} for a detailed descriptions and \cite{Timmes2013}). This
EoS includes contributions of (photon) radiation, nuclei, electrons and
positrons and is well adapted to large range of stellar environments. The
radiation is treated as a blackbody in local thermal equilibrium and the nuclei
are modeled by the ideal gas law. For computational efficiency, the electrons
and positrons are treated in a tabular manner with a thermodynamically
consistent interpolation procedure.

The white dwarf model is fully characterized by specifying the central density,
the chemical composition and the thermodynamic equilibrium. We set the central
density $\rho = \SI{2e9}{g \per cm\cubed}$ and temperature $T = \SI{5e8}{K}$. We
assume a constant specific entropy and set the composition to half carbon
$^{12}\mathrm{C}$ and half oxygen $^{12}\mathrm{O}$. Then the model can be
constructed by simple numerical integration of the self-gravitating hydrostatic
equilibrium equations in spherical symmetry. We refer to \cite{Kaeppeli2014199}
for the detailed procedure.

The one-dimensional white dwarf profile is then mapped onto the two-dimensional
computational domain $[-L,L]^2$ with $L = \SI{1e8}{cm}$. The velocity is set to
zero. The same hydrostatic extrapolation boundary conditions are used as in
\cref{subsec:numex_2dpoly}.

\subsubsection{Well-balanced property}
We evolve the hydrostatic equilibrium without perturbation on a grid with $N =
128^2$ cells until time $t = \SI{1}{s}$. The unbalanced scheme has
$err_{eq,1}(\rho) = \SI{3.21e4}{g / cm^3}$ and $err_{eq,1}(E) =
\SI{2.79e22}{erg/cm^3}$. The well-balanced scheme is confirmed to be exact up to
machine precision, with errors of $err_{eq,1}(\rho) = \SI{5.59e-06}{g / cm^3}$
and $err_{eq,1}(E) = \SI{7.45e+12}{erg/cm^3}$. This shows that equation
\cref{eq:nm_md_0080} can be solved numerically and the solution is effectively
unique. If the iterative procedure where to find a different equilibrium in any
cell, it would be very unlikely that the resulting scheme would be
well-balanced.

\subsubsection{Wave propagation} \label{sec:white_dwarf_smooth}
 Next we add a small Gaussian pressure perturbation at the origin, i.e.\
\begin{align}
 p_0(\bx) = (1 + A \exp(-\abs{\bx}^2 / 2 b^2))\ p_{eq}(\bx),
\end{align}
 with $A = 10^{-3}$ and $b = \SI{1e7}{cm}$. The solution is
evolved to $t = \SI{7.32e-2}{s}$ on $N = 128^2$. A scatter plot of the solution
is shown in \cref{fig:white_dwarf_smooth}. The scatter in the well-balanced
scheme is significantly reduced compared to the unbalanced scheme. Unlike the
unbalanced solution, the well-balanced solution remains constant ahead of the
perturbation and returns to rest after the perturbation has passed.

The reference solution is computed using a one-dimensional, cylindrically
symmetric, well-balanced finite volume code with a resolution of $N =
\num{8192}$.

\begin{figure}[htbp]
  \begin{center}
    \includegraphics[width=0.45\columnwidth]{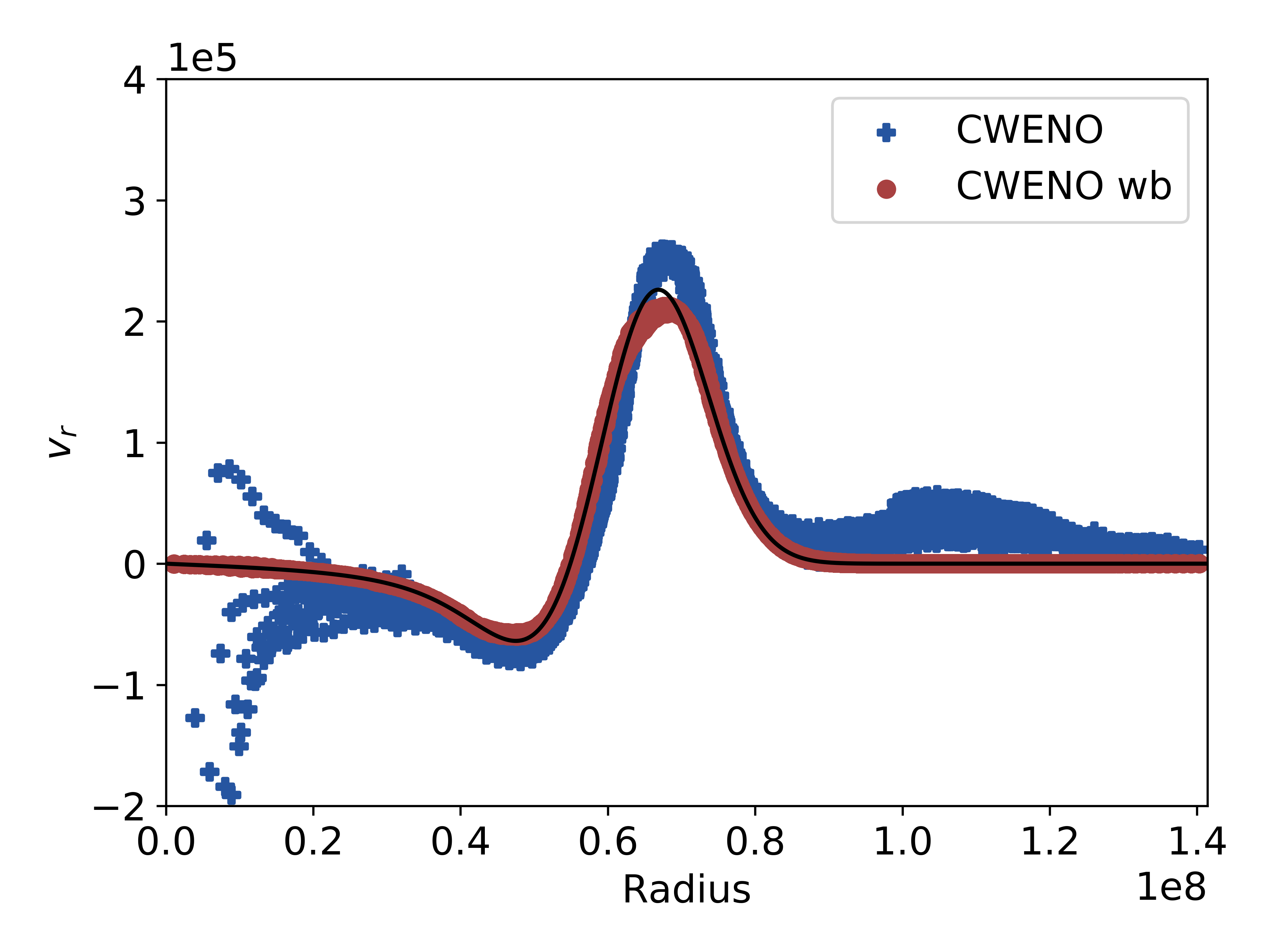}
    \includegraphics[width=0.45\columnwidth]{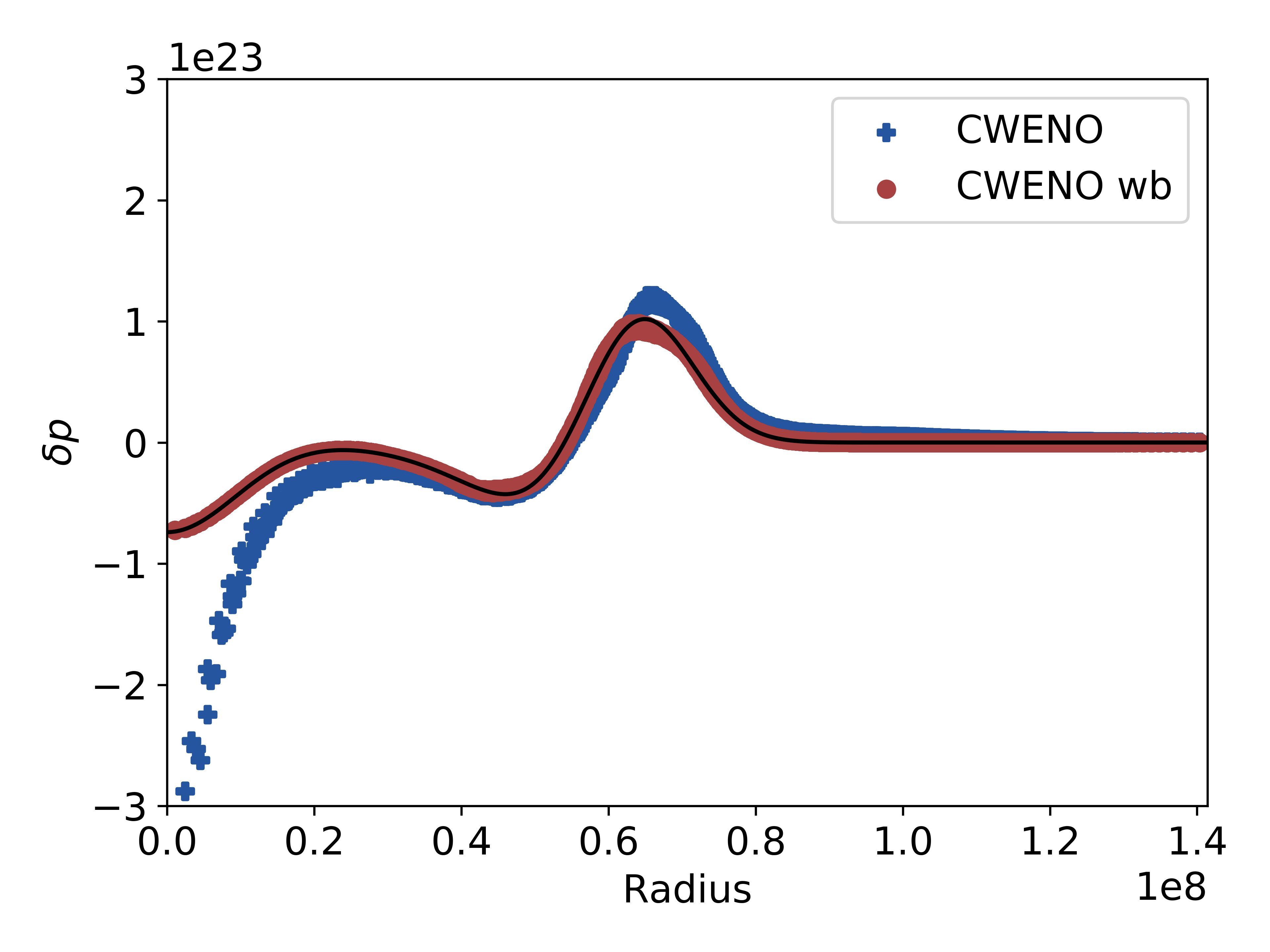}
  \end{center}
  \vspace{-0.2in}
  \caption{Snapshots of the two-dimensional white-dwarf, see
\cref{sec:white_dwarf_smooth} with a small perturbation. The simulation is
performed with $N=128^2$ cells and run up to time $t = \SI{7.32e-2}{s}$. The
radial velocity is shown on the left, the pressure perturbation on the right.
The size of the initial perturbation was chosen such that the main feature of the
perturbation is resolved similarly well in both solvers. However, the
well-balanced scheme has significantly less scatter and shows no deviation from
equilibrium ahead of the perturbation. Furthermore, the well-balanced scheme
returns to rest after the wave has moved away from the center of the domain,
i.e.\ $r =\SI{0}{cm}$ on the plot. The units in the plot are CGS.
}
  \label{fig:white_dwarf_smooth}
\end{figure}


\section{Conclusion}
\label{sec:conc}
We presented a novel well-balanced, high-order finite volume scheme for Euler
equations with gravity. We are able to well-balance a large class of
astrophysically relevant hydrostatic equilibria without imposing the exact
equilibrium apriori. Rather, we only assume some thermodynamic information about
the equilibrium, e.g. constant entropy, and solve for the equilibrium in every
time step. Since the equilibrium defined by
\begin{align}
  \grad{p} = - \rho \grad{\phi}
\end{align}
is only a mechanical equilibrium, it seems natural that some additional
information about the thermodynamic nature of the equilibrium must always be
imposed.

The important features of the proposed scheme are:
\begin{itemize}
  \item Its independence of a particular form of equation of state. This scheme can
    handle arbitrary equations of state including tabulated ones, as shown in the
    final numerical experiment.

  \item Its independence of a particular gravitational potential. The only
    requirement is that the gravitation potential and its gradient can be
    evaluated at apriori known locations in the computation domain. In fact the
    gravitation source term does not need to be constant in time. Therefore this
    scheme can also be applied to simulations which include self-gravity. Such
    simulations may benefit from well-balancing if the initial conditions are at
    rest and perturbed by some other means, such as a heating source term.

  \item Its modular nature. The scheme clearly describes how any reconstruction
    procedure can be made well-balanced. Therefore, the proposed scheme can be
    extended to arbitrary orders in a straightforward manner.

  \item Its local nature. The well-balancing is local to each cell. In
    particular it does not change the stencil required to update the cell.
\end{itemize}

The numerical experiments have shown that the scheme is high-order accurate for
flows both near and far away from hydrostatic equilibrium. In fact, the
numerical results suggest that the scheme is no worse on large perturbations
than the equivalent unbalanced scheme. We have also shown that the schemes are
stable in the presence of shocks. Furthermore, the smooth tests show that the
well-balanced scheme preserves radial symmetry much better than the unbalanced
scheme. Furthermore, the well-balanced solutions do not cause any changes in the
part of the domain the perturbation has not reached yet. Additionally, the
well-balanced scheme returns to rest after the perturbation has passed over some
region in the domain, while the unbalanced scheme does neither. These tests were
performed under a variety of different conditions, i.e.\ in one dimension for
constant gravity, in two dimensions for non-grid aligned gravity with both the
ideal gas law and a complex multiphysics equation of state.

Our scheme only affects the reconstruction procedure and the numerical source
term. Therefore, large parts of an existing finite volume code would remain
unaffected by adding our well-balancing. By reusing the existing unbalanced
reconstruction procedure for the perturbation the cost of implementing our
scheme is further reduced. These very localized and modular changes ensure that
our method can be used to well-balance a variety of different existing finite
volume schemes with minimal effort.


\section*{Acknowledgments}
The work was supported by the Swiss National Science Foundation (SNSF) under
grant 200021-169631.
The authors also acknowledge the use of computational resources provided by
the Swiss SuperComputing Center (CSCS), under the allocation grant s661, s665,
s667 and s744.
We acknowledge the computational resources provided by the EULER cluster of
ETHZ.


\bibliographystyle{elsarticle-num.bst}

\let\jnl=\rm
\def\aap{\jnl{A\&A}}               
\def\apj{\jnl{ApJ}}                
\def\apjl{\jnl{ApJ}}               
\def\apjs{\jnl{ApJS}}              
\def\physrep{\jnl{Phys.~Rep.}} 

\bibliography{biblio.bib}


\appendix
\section{Equilibrium reconstruction for the ideal gas law}
\label{sec:appendix01}
In the ideal gas case, it can be shown that a unique equilibrium exists which
matches the cell-averages, i.e.\ satisfies \cref{eq:nm_md_0120} (in one
dimension) and \cref{eq:nm_1d_wbrec_0020} (in two dimensions).

In a first step the system is reduced to a single nonlinear equation in one
unknown. To this end, we write the ideal gas law in the polytropic form
\begin{equation}
  \label{eq:nm_1d_wbrec_0040_appendix}
  p = p(K,\rho) = K \rho^{\gamma}
  ,
\end{equation}
where $K = K(s)$ is a function of entropy $s$ alone and $\gamma$ is the ratio
of specific heats.
Then the equilibrium density and internal energy density can be expressed as
functions of the constant $K_{0,i}$ and the enthalpy at cell center $h_{0,i}$:
\begin{equation}
  \label{eq:nm_1d_wbrec_0050_appendix}
  \begin{aligned}
  \rho_{eq,i}  (x) & = \left( \frac{1}{K_{0,i}} \frac{\gamma - 1}{\gamma}
                              h_{eq,i}(x)
                       \right)^{\frac{1}{\gamma - 1}} \\
  \rho e_{eq,i}(x) & = \frac{1}{\gamma - 1}
                       \left( \frac{1}{K_{0,i}} \right)^{\frac{1}{\gamma - 1}}
                       \left( \frac{\gamma - 1}{\gamma}
                              h_{eq,i}(x)
                       \right)^{\frac{\gamma}{\gamma - 1}}
  .
  \end{aligned}
\end{equation}
By plugging the latter into \eqref{eq:nm_1d_wbrec_0020}, one obtains a
single equation for $h_{0,i}$
\begin{equation}\label{eq:a3}
  \mrhoe_{i} = \frac{\mrho_{i}}{\gamma - 1}
               \frac{\sum_{j=1}^{N_{q}} w_{j}
                       \left( \frac{\gamma - 1}{\gamma}
                              \left( h_{0,i} + \phi_{i} - \phi(x_{j}) \right)
                       \right)^{\frac{\gamma}{\gamma - 1}}}
                    {\sum_{j=1}^{N_{q}} w_{j}
                       \left( \frac{\gamma - 1}{\gamma}
                              \left( h_{0,i} + \phi_{i} - \phi(x_{j}) \right)
                       \right)^{\frac{1}{\gamma - 1}}} =: f(h_{0,i})
\end{equation}
and the constant $K_{0,i}$ is simply given by
\begin{equation}
  K_{0,i} = \left[ \frac{1}{\Delta x ~ \mrho_{i}}
                   \sum_{j=1}^{N_{q}} w_{j}
                       \left( \frac{\gamma - 1}{\gamma}
                              \left( h_{0,i} + \phi_{i} - \phi(x_{j}) \right)
                       \right)^{\frac{1}{\gamma - 1}}
            \right]^{\gamma - 1}
  .
\end{equation}
To show that \cref{eq:a3} has a unique solution we show that it is monotone.
Therefore, we differentiate $f$ and find
\begin{align}
  f^\prime(h_{0,i})
 = \frac{\mrho_{i}}{\gamma - 1}
 \left( 1
   - \frac{1}{\gamma}
   \frac{
    \sum_{j=1}^{N_{q}} w_{j}
           \left(
                   h_{0, i} + \phi_{i} - \phi(x_{j})
           \right)^{\frac{\gamma}{\gamma - 1}}
    \cdot
    \sum_{j=1}^{N_{q}} w_{j}
           \left(
                   h_{0, i} + \phi_{i} - \phi(x_{j})
           \right)^{\frac{2 - \gamma}{\gamma - 1}}
   }
   {\left(\sum_{j=1}^{N_{q}} w_{j}
       \left(
              h_{0,i} + \phi_{i} - \phi(x_{j})
       \right)^{\frac{1}{\gamma - 1}}\right)^2
   }
  \right).
\end{align}
Clearly, the second term is positive, and if it where less than $\gamma$ the
derivative of $f$ would be positive, everywhere, and therefore $f$ would be a
strictly monotone function. If within every cell $\phi$ does not vary too much, this turns
out to be true and can be proven as follows.

Let
\begin{align}
 h_{max,i} &= h_{0,i} + \max_{x\in [x_{i-\half}, x_{i+\half}]}{\phi_i - \phi(x)}
  \\
 h_{min,i} &= h_{0,i} + \min_{x\in [x_{i-\half}, x_{i+\half}]}{\phi_i - \phi(x)}
\end{align}
then for $1 < \gamma \leq 2$ we find
\begin{align}
   \frac{
    \sum_{j=1}^{N_{q}} w_{j}
           \left(
                  h_{0, i} + \phi_{i} - \phi(x_{j})
           \right)^{\frac{\gamma}{\gamma - 1}}
    \cdot
    \sum_{j=1}^{N_{q}} w_{j}
           \left(
                  h_{0,i} + \phi_{i} - \phi(x_{j})
           \right)^{\frac{2 - \gamma}{\gamma - 1}}
   }
   {\left(\sum_{j=1}^{N_{q}} w_{j}
       \left(
              h_{0,i} + \phi_{i} - \phi(x_{j})
       \right)^{\frac{1}{\gamma - 1}}\right)^2
   }
  \\
  \leq
   \frac{
    \sum_{j=1}^{N_{q}} w_{j}
           h_{max,i}^{\frac{\gamma}{\gamma - 1}}
    \cdot
    \sum_{j=1}^{N_{q}} w_{j}
           h_{max,i}^{\frac{2 - \gamma}{\gamma - 1}}
   }
   {\left(\sum_{j=1}^{N_{q}} w_{j}
       h_{min,i}^{\frac{1}{\gamma - 1}}\right)^2
   }
  = \frac{h_{max,i}^2}{h_{min,i}^2}
  .
\end{align}
Therefore, under the condition that
\begin{align}
  \frac{h_{max,i}}{h_{min,i}} < \gamma^{\half}
\end{align}
$f$ has a unique solution. By a very similar estimate we find that for $\gamma >
2$ a unique solution exists provided
\begin{align}
  \frac{h_{max,i}}{h_{min,i}} < \gamma^{\frac{\gamma-1}{\gamma}}.
\end{align}

\end{document}